\documentclass{amsart}

\usepackage[T1]{fontenc}
\usepackage{amsmath, amsfonts, amssymb, amsthm, mathrsfs, wasysym, 
			a4wide, graphics, graphicx, xcolor, overpic, pdflscape, hvfloat, minibox, 
			url, hyperref, hypcap, %
			xargs, xifthen, %
			enumerate,
			multicol, multirow, array, blkarray}
\usepackage{}
\hypersetup{colorlinks=true, citecolor=darkblue, linkcolor=darkblue}
\usepackage[all]{xy}
\usepackage[bottom]{footmisc}
\usepackage{tikz}
\usepackage{tkz-graph}
\usetikzlibrary{trees, decorations, decorations.markings, shapes, arrows, matrix, calc, fit, intersections, patterns, angles}
\graphicspath{{figures/}}
\makeatletter\def\input@path{{figures/}}\makeatother
\usepackage{caption}
\newtheorem{theorem}{Theorem}
\newtheorem{corollary}[theorem]{Corollary}
\newtheorem{proposition}[theorem]{Proposition}
\newtheorem{lemma}[theorem]{Lemma}
\newtheorem{conjecture}[theorem]{Conjecture}
\newtheorem*{theorem*}{Theorem}

\theoremstyle{definition}

\newtheorem{example}[theorem]{Example}
\newtheorem{remark}[theorem]{Remark}

\newcommand{\R}{\mathbb{R}} 
\newcommand{\N}{\mathbb{N}} 
\newcommand{\TP}{\mathbb{TP}} 
\newcommand{\fS}{\mathfrak{S}} 
\newcommand{\cI}{\mathcal{I}} 
\renewcommand{\b}[1]{\mathbf{#1}} 

\newcommand{\set}[2]{\left\{ #1 \;\middle|\; #2 \right\}} 
\newcommand{\bigset}[2]{\big\{ #1 \;\big|\; #2 \big\}} 
\newcommand{\ssm}{\smallsetminus} 
\newcommand{\one}{{1\!\!1}} 
\newcommand{\eqdef}{\mbox{\,\raisebox{0.2ex}{\scriptsize\ensuremath{\mathrm:}}\ensuremath{=}\,}} 

\DeclareMathOperator{\conv}{conv} 

\newcommand{\fref}[1]{Figure~\ref{#1}} 
\newcommand{\ie}{\textit{i.e.}~} 
\definecolor{darkblue}{rgb}{0,0,0.7} 
\definecolor{green}{RGB}{57,181,74} 
\definecolor{violet}{RGB}{147,39,143} 
\newcommand{\darkblue}{\color{darkblue}} 
\newcommand{\defn}[1]{\textsl{\darkblue #1}} 

\usepackage{todonotes}

\makeatletter
\def\part{\@startsection{part}{1}%
\z@{.7\linespacing\@plus\linespacing}{.8\linespacing}%
{\LARGE\sffamily\centering}}
\@addtoreset{section}{part}
\makeatother

\makeatletter
\def\l@section{\@tocline{1}{2pt}{0pc}{}{}}
\makeatother
\let\oldtocpart=\tocpart
\renewcommand{\tocpart}[2]{\bf\large\oldtocpart{#1}{#2}}
\let\oldtocsection=\tocsection
\renewcommand{\tocsection}[2]{\bf\oldtocsection{#1}{#2}}

\newcommand{\llb}{[\![} 
\newcommand{\rrb}{]\!]} 
\newcommand{\signature}{\varepsilon}
\newcommand{\sq}{{\scalebox{.5}{$\square$}}}
\newcommandx{\polygonS}[1][1=\signature]{\mathrm{P}_{\!\sq}^{#1}}
\newcommandx{\polygonC}[1][1=\signature]{\mathrm{P}_{{\!\bullet\!\!\!-\!\circ}}^{#1}}
\newcommandx{\G}[1][1 = \signature]{\mathrm{G}^{#1}}
\newcommandx{\GIJ}[3][1 = I_\bullet, 2 = J_\circ, 3 = \signature]{\mathrm{G}_{#1,#2}^{#3}}
\newcommandx{\IFGIJ}[3][1 = I_\bullet, 2 = J_\circ, 3 = \signature]{\mathrm{F}_{#1,#2}^{#3}}
\newcommandx{\Lat}[1][1 = \signature]{\mathrm{L}^{#1}}
\newcommandx{\LatIJ}[3][1 = I_\bullet, 2 = J_\circ, 3 = \signature]{\mathrm{L}_{#1,#2}^{#3}}
\newcommand{\tree}{\mathsf{t}} 

\newcommandx{\treeMin}{\mathsf{tmin}} 
\newcommandx{\treeMax}{\mathsf{tmax}} 
\newcommandx{\treeMinIJ}[3][1 = I_\bullet, 2 = J_\circ, 3 = \signature]{\treeMin_{#1,#2}^{#3}} 
\newcommandx{\treeMaxIJ}[3][1 = I_\bullet, 2 = J_\circ, 3 = \signature]{\treeMax_{#1,#2}^{#3}} 
\newcommand{\forest}{\mathsf{f}} 
\newcommand{\cat}{\mathrm{cat}} 
\newcommand{\subpoly}{\mathrm{U}} 
\newcommandx{\subpolyIJ}[2][1 = I_\bullet, 2 = J_\circ]{\subpoly_{#1,#2}} 
\newcommandx{\complexIJ}[3][1 = I_\bullet, 2 = J_\circ, 3 = \signature]{\mathrm{C}_{#1,#2}^{#3}} 
\newcommandx{\CambTriang}[1][1 = \signature]{\mathcal{T}^{#1}} 
\newcommandx{\CambTriangIJ}[3][1 = I_\bullet, 2 = J_\circ, 3 = \signature]{\mathcal{T}_{#1,#2}^{#3}} 
\newcommand{\canopy}{\mathrm{can}} 
\newcommandx{\Asso}[4][1 = I_\bullet, 2 = J_\circ, 3 = \signature, 4 = h]{\mathsf{Asso}_{#1,#2}^{#3}(#4)} 
\newcommand{\newleftrightarrow}{\mathbin{\tikz [baseline=0.2em] \draw [<->] (-.4em,.4em) -- (.4em,.4em);}} 
\newcommand{\newupdownarrow}{\mathbin{\tikz [thin, baseline=-0.2em] \draw [<->] (0em,-.3em) -- (0em,.3em);}} 
\newcommand{\vertmirror}[1]{#1^{\newleftrightarrow}} 
\newcommand{\horimirror}[1]{#1^{\!\ \newupdownarrow\!\ }} 
\newcommand{\SSS}{\reflectbox{$\mathsf{Z}$}} 
\newcommand{\ZZZ}{\mathsf{Z}} 

\title{Cambrian triangulations and their tropical realizations}

\thanks{Partially supported by the French ANR grants SC3A~(15\,CE40\,0004\,01) and CAPPS~(17\,CE40\,0018).}

\author{Vincent Pilaud}
\address{CNRS \& LIX, \'Ecole Polytechnique, Palaiseau}
\email{vincent.pilaud@lix.polytechnique.fr}
\urladdr{\url{http://www.lix.polytechnique.fr/~pilaud/}}

\begin{document}

\begin{abstract}
This paper develops a Cambrian extension of the work of C.~Ceballos, A.~Padrol and C.~Sarmiento on $\nu$-Tamari lattices and their tropical realizations.
For any signature~$\signature \in \{\pm\}^n$, we consider a family of $\signature$-trees in bijection with the triangulations of the $\signature$-polygon. These $\signature$-trees define a flag regular triangulation~$\CambTriang$ of the subpolytope~${\conv \set{(\b{e}_{i_\bullet}, \b{e}_{j_\circ})}{0 \le i_\bullet < j_\circ \le n+1}}$ of the product of simplices~$\triangle_{\{0_\bullet, \dots, n_\bullet\}} \times \triangle_{\{1_\circ, \dots, (n+1)_\circ\}}$. The oriented dual graph of the triangulation~$\CambTriang$ is the Hasse diagram of the (type~$A$) $\signature$-Cambrian lattice of N.~Reading. For any~$I_\bullet \subseteq \{0_\bullet, \dots, n_\bullet\}$ and~$J_\circ \subseteq \{1_\circ, \dots, (n+1)_\circ\}$, we consider the restriction~$\CambTriangIJ$ of the triangulation~$\CambTriang$ to the face~$\triangle_{I_\bullet} \times \triangle_{J_\circ}$. Its dual graph is naturally interpreted as the increasing flip graph on certain $(\signature, I_\bullet, J_\circ)$-trees, which is shown to be a lattice generalizing in particular the $\nu$-Tamari lattices in the Cambrian setting. Finally, we present an alternative geometric realization of~$\CambTriangIJ$ as a polyhedral complex induced by a tropical hyperplane arrangement.
\end{abstract}

\vspace*{-.8cm}

\maketitle


The Tamari lattice is a fundamental structure on Catalan objects (such as triangulations, binary trees, or Dyck paths). Introduced by D.~Tamari in~\cite{Tamari}, it has been extensively studied and extended in several directions, see in particular~\cite{TamariFestschrift} and the references therein. Our objective is to explore the connection between two generalizations of the Tamari lattice: the (type~$A$) Cambrian lattices of N.~Reading~\cite{Reading-CambrianLattices} and the $\nu$-Tamari lattices of L.-F.~Pr\'eville-Ratelle and X.~Viennot~\cite{PrevilleRatelleViennot}. These two generalizations have strong algebraic roots, in connection to cluster algebras~\cite{FominZelevinsky-ClusterAlgebrasI, FominZelevinsky-ClusterAlgebrasII} and multivariate diagonal harmonics~\cite{Bergeron-multivariateDiagonalCoinvariantSpaces, BergeronPrevilleRatelle}.

This paper heavily relies on the work of C.~Ceballos, A.~Padrol and C.~Sarmiento on the geometry of $\nu$-Tamari lattices~\cite{CeballosPadrolSarmiento}. They start from a family of non-crossing alternating trees in bijection with the triangulations of the $(n+2)$-gon. These trees define a flag regular triangulation~$\CambTriang[]$ of the subpolytope $U \eqdef {\conv \set{(\b{e}_{i_\bullet}, \b{e}_{j_\circ})}{0 \le i_\bullet < j_\circ \le n+1}}$ of the product of simplices~${\triangle_{\{0_\bullet, \dots, n_\bullet\}} \times \triangle_{\{1_\circ, \dots, (n+1)_\circ\}}}$. The dual graph of this triangulation~$\CambTriang[]$ is the Hasse diagram of the Tamari lattice. Note that this interpretation of the Tamari lattice as the dual graph of the non-crossing triangulation is ubiquitous in the literature as discussed in~\cite[Sect.~1.4]{CeballosPadrolSarmiento}. For any subsets~$I_\bullet \subseteq \{0_\bullet, \dots, n_\bullet\}$ and~$J_\circ \subseteq \{1_\circ, \dots, (n+1)_\circ\}$, they consider the restriction~$\CambTriangIJ[I_\bullet][J_\circ][]$ of the triangulation~$\CambTriang[]$ to the face~$\triangle_{I_\bullet} \times \triangle_{J_\circ}$. The simplices of~$\CambTriangIJ[I_\bullet][J_\circ][]$ correspond to certain non-crossing alternating $(I_\bullet, J_\circ)$-trees which are in bijection with Dyck paths above a fixed path~$\nu(I_\bullet, J_\circ)$. Moreover, the dual graph of~$\CambTriangIJ[I_\bullet][J_\circ][]$ is the flip graph on $(I_\bullet, J_\circ)$-trees, isomorphic to the \mbox{$\nu(I_\bullet, J_\circ)$-Tamari} poset of~\cite{PrevilleRatelleViennot}. This poset actually embeds as an interval of the classical Tamari lattice and is therefore itself a lattice. This interpretation provides three geometric realizations of the $\nu(I_\bullet, J_\circ)$-Tamari lattice~\cite[Thm.~1.1]{CeballosPadrolSarmiento}: as the dual of the regular triangulation~$\CambTriangIJ[I_\bullet][J_\circ][]$, as the dual of a coherent mixed subdivision of a generalized permutahedron, and as the edge graph of a polyhedral complex induced by a tropical hyperplane arrangement.

Our objective is to extend this approach in the type~$A$ Cambrian setting. For any signature~${\signature \in \{\pm\}^n}$, we consider a family of $\signature$-trees in bijection with the triangulations of the $\signature$-polygon. These $\signature$-trees define a flag regular triangulation~$\CambTriang$ of~$U$ whose dual graph is the Hasse diagram of the (type~$A$) $\signature$-Cambrian lattice of N.~Reading~\cite{Reading-CambrianLattices}. In contrast to the classical Tamari case (obtained when $\signature = {-}^n$), we are not aware that this triangulation of~$U$ was considered earlier in the literature and the proof of its regularity is a little more subtle in the Cambrian case. For any~$I_\bullet \subseteq \{0_\bullet, \dots, n_\bullet\}$ and~$J_\circ \subseteq \{1_\circ, \dots, (n+1)_\circ\}$, we then consider the restriction~$\CambTriangIJ$ of the triangulation~$\CambTriang$ to the face~$\triangle_{I_\bullet} \times \triangle_{J_\circ}$. Its simplices correspond to certain~$(\signature, I_\bullet, J_\circ)$-trees and its dual graph is the increasing flip graph on these $(\signature, I_\bullet, J_\circ)$-trees. Our main combinatorial result is that this increasing flip graph is still an interval of the $\signature$-Cambrian lattice in general. The proof is however more involved than in the classical case ($\signature = {-}^n$) since this interval does not anymore correspond to a descent class in general. Finally, we mimic the method of~\cite[Sect.~5]{CeballosPadrolSarmiento} to obtain an alternative geometric realization of~$\CambTriangIJ$ as a polyhedral complex induced by a tropical hyperplane arrangement.

\section{$(\signature, I_\bullet, J_\circ)$-trees and the $(\signature, I_\bullet, J_\circ)$-complex}

This section defines two polygons and certain families of trees associated to a signature~${\signature \in \{\pm\}^n}$.

\subsection{Two $\signature$-polygons}

We consider three decorated copies of the natural numbers: the squares~$\N_\sq$, the blacks~$\N_\bullet$ and the whites~$\N_\circ$.
For~$n \in \N$, we use the standard notation $[n] \eqdef \{1, \dots, n\}$ and define~${\llb n ] \eqdef \{0, \dots, n\}}$, ${[ n \rrb \eqdef \{1, \dots, n+1\}}$ and~$\llb n \rrb \eqdef \{0, \dots, n+1\}$.
We write~$[n_\sq]$, $[n_\bullet]$, $[n_\circ]$ and so on for the decorated versions of these intervals.
Fix a signature~$\signature \in \{\pm\}^n$. We consider two convex polygons associated to the signature~$\signature$ as follows:
\begin{itemize}
\item a $(n+2)$-gon~$\polygonS$ with square vertices labeled by~$\llb n_\sq \rrb$ from left to right and where vertex~$i_\sq$ is above the segment~$(0_\sq, (n+1)_\sq)$ if~$\signature_i = {+}$ and below it if~$\signature_i = {-}$.
\item a $(2n+2)$-gon~$\polygonC$ with black or white vertices, obtained from $\polygonS$ by replacing the square vertex~$0_\sq$ (resp.~$(n+1)_\sq$) by the black vertex~$0_\bullet$ (resp.~white vertex~$(n+1)_\circ$), and splitting each other square vertex~$i_\sq$ into a pair of white and black vertices~$i_\circ$ and~$i_\bullet$ (such that the vertices of~$\polygonC$ are alternatively colored black~and~white). The black (resp.~white) vertices of~$\polygonC$ are labeled by~$\llb n_\bullet ]$ (resp.~$[ n_\circ \rrb$) from left to right.
\end{itemize}
Examples of these polygons are represented in \fref{fig:polygons} for the signature~$\signature = {-}{+}{+}{-}{+}{-}{-}{+}$.

\begin{figure}[h]
	\capstart
	\centerline{\includegraphics[scale=.9]{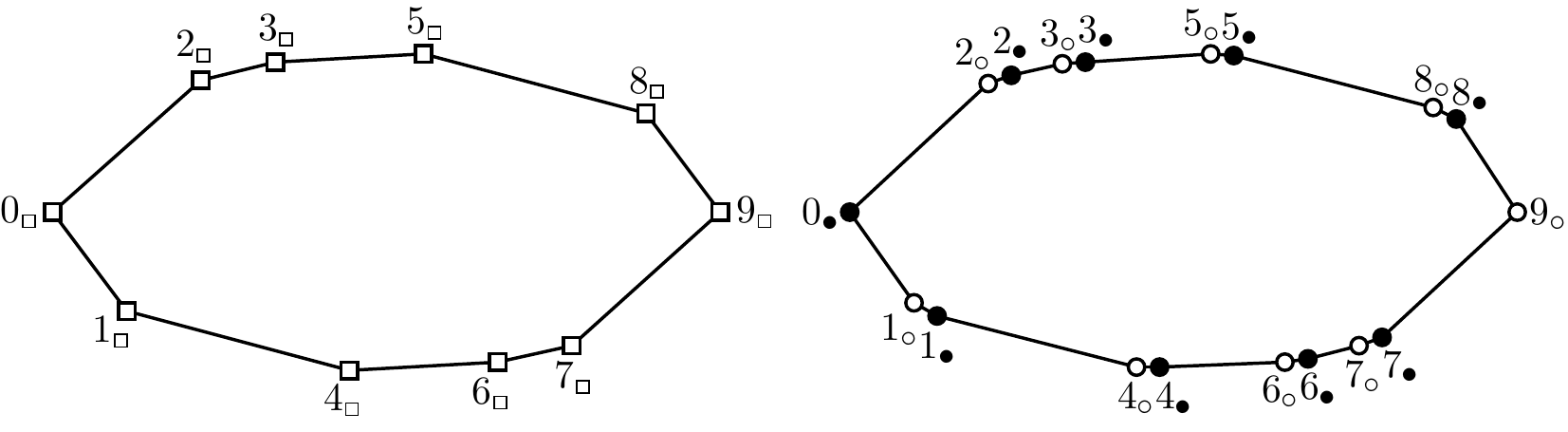}}
	\caption{The polygons~$\polygonS$ (left) and~$\polygonC$ (right) for the signature~$\signature = {-}{+}{+}{-}{+}{-}{-}{+}$.}
	\label{fig:polygons}
\end{figure}

\subsection{$(\signature, I_\bullet, J_\circ)$-trees}

All throughout the paper, we consider~$I_\bullet \subseteq \llb n_\bullet ]$ and~$J_\circ \subseteq [ n_\circ \rrb$ and we always assume that~${\min(I_\bullet) < \min(J_\circ)}$ and~${\max(I_\bullet) < \max(J_\circ)}$.
Consider the graph~$\GIJ$ with vertices~$I_\bullet \cup J_\circ$ and edges~$\set{(i_\bullet, j_\circ)}{i_\bullet \in I_\bullet, \, j_\circ \in J_\circ, \, i_\bullet < j_\circ}$.
Note that this graph is geometric: its vertices are considered as vertices of~$\polygonC$ and its edges are considered as straight edges in~$\polygonC$.
A subgraph of~$\GIJ$ is \defn{non-crossing} if no two of its edges cross in their interior.

\begin{proposition}
\label{prop:trees}
Any maximal non-crossing subgraph of~$\GIJ$ is a spanning tree of~$\GIJ$.
\end{proposition}

\begin{proof}
The proof works by induction on~$|I_\bullet|+|J_\circ|$. The result is immediate when~$|I_\bullet| = |J_\circ| = 1$. Assume now for instance that~$|I_\bullet| > 1$ (the case $|I_\bullet| = 1$ and~$|J_\circ| > 1$ is similar). Let~$i_\bullet \eqdef \max(I_\bullet)$ and~$j_\circ \eqdef \min \big( \set{j_\circ \in J_\circ}{i_\bullet < j_\circ \text{ and } \signature_i = \signature_j} \cup \{\max(J_\circ)\} \big)$. Note that our choice of~$j_\circ$ ensures that~$(i_\bullet, j_\circ)$ is a boundary edge of~$\conv(I_\bullet \cup J_\circ)$. Moreover, any edge of~$\GIJ$ incident to~$i_\bullet$ is of the form~$(i_\bullet, j_\circ')$ for~$i_\bullet < j_\circ'$ while any edge of~$\GIJ$ incident to~$j_\circ$ is of the form~$(i'_\bullet, j_\circ)$ for~$i'_\bullet \le i_\bullet$ (by maximality of~$i_\bullet$). Therefore, all edges of~$\GIJ \ssm \{(i_\bullet, j_\circ)\}$ incident to~$i_\bullet$ cross all edges of~$\GIJ \ssm \{(i_\bullet, j_\circ)\}$ incident to~$j_\circ$. Consider now a maximal non-crossing subgraph~$\tree$ of~$\GIJ$. Then~$\tree$ contains the edge~$(i_\bullet, j_\circ)$ (since~$\tree$ is maximal) and either~$i_\bullet$ or~$j_\circ$ is a leaf in~$\tree$ (since~$\tree$ is non-crossing). Assume for example that~$i_\bullet$ is a leaf and let~$I_\bullet' \eqdef I_\bullet \ssm \{i_\bullet\}$. Then~$\tree \ssm \{(i_\bullet, j_\circ)\}$ is a maximal non-crossing subgraph of~$\GIJ[I_\bullet']$ (the maximality is ensured from the fact that~$(i_\bullet, j_\circ)$ is a boundary edge of~$\conv(I_\bullet \cup J_\circ)$). By induction, $\tree \ssm \{(i_\bullet, j_\circ)\}$ is thus a spanning tree of~$\GIJ[I_\bullet']$, so that~$\tree$ is a spanning tree of~$\GIJ$.
\end{proof}

In accordance to Proposition~\ref{prop:trees}, we define a \defn{$(\signature, I_\bullet, J_\circ)$-forest} to be a non-crossing subgraph of~$\GIJ$, and a \defn{$(\signature, I_\bullet, J_\circ)$-tree} to be a maximal $(\signature, I_\bullet, J_\circ)$-forest. Note that a $(\signature, I_\bullet, J_\circ)$-tree has~$|I_\bullet| + |J_\circ| - 1$ edges. Examples can be found in \fref{fig:forestTree}.

\begin{figure}[t]
	\capstart
	\centerline{\includegraphics[scale=.9]{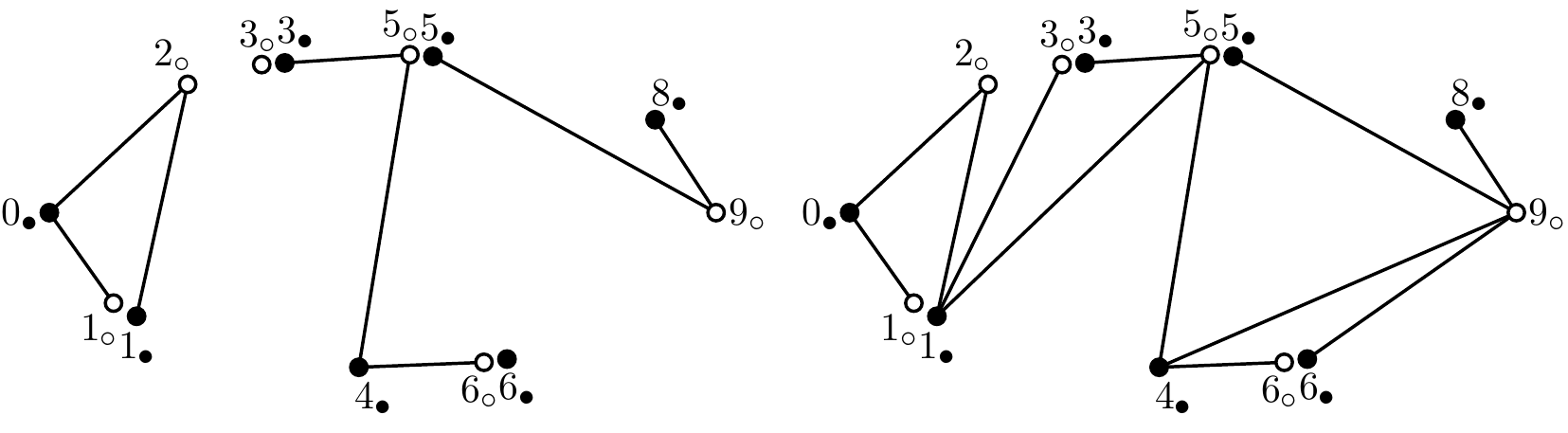}}
	\caption{A $(\signature, I_\bullet, J_\circ)$-forest (left) and a $(\signature, I_\bullet, J_\circ)$-tree (right) for~$\signature = {-}{+}{+}{-}{+}{-}{-}{+}$, ${I_\bullet = \llb 8_\bullet ] \ssm \{2_\bullet, 7_\bullet\}}$ and~$J_\circ = [ 8_\circ \rrb \ssm \{4_\circ, 7_\circ, 8_\circ\}$.}
	\label{fig:forestTree}
\end{figure}

\subsection{The $(\signature, I_\bullet, J_\circ)$-complex}

We call \defn{$(\signature, I_\bullet, J_\circ)$-complex}~$\complexIJ$ the clique complex of the graph of non-crossing edges of~$\GIJ$. In other words, its ground set is the edge set of~$\GIJ$, its faces are the $(\signature, I_\bullet, J_\circ)$-forests, and its facets are the $(\signature, I_\bullet, J_\circ)$-trees.

We say that an edge~$(i_\bullet, j_\circ)$ of~$\GIJ$ is \defn{irrelevant} if it is not crossed by any other edge of~$\GIJ$ (\ie, there is no~$i'_\bullet \in I_\bullet$ and~$j'_\circ \in J_\circ$ separated by~$(i_\bullet, j_\circ)$ and such that~$i'_\bullet < j'_\circ$). In particular, all edges of~$\GIJ$ on the boundary of~$\conv(I_\bullet \cup J_\circ)$ are irrelevant. Note that all $(\signature, I_\bullet, J_\circ)$-trees contain all irrelevant edges of~$\GIJ$, so that the $(\signature, I_\bullet, J_\circ)$-complex $\complexIJ$ is a pyramid over the irrelevant edges of~$\GIJ$.

Although the next statement will directly follow from Proposition~\ref{prop:CambrianTriangulations}, we state and prove it here to develop our understanding on the $(\signature, I_\bullet, J_\circ)$-complex. Recall that a simplicial complex is a \defn{pseudomanifold} when it is \defn{pure} (all its maximal faces have the same dimension) and \defn{thin} (any codimension~$1$ face is contained in at most two facets).

\begin{proposition}
\label{prop:pseudomanifold}
The $(\signature, I_\bullet, J_\circ)$-complex~$\complexIJ$ is a pseudomanifold.
\end{proposition}

\begin{proof}
The $(\signature, I_\bullet, J_\circ)$-complex~$\complexIJ$ is pure of dimension~$|I_\bullet| + |J_\circ| - 1$ since all its maximal faces are spanning trees of~$\GIJ$.
To show that it is thin, assume by contradiction that a codimension~$1$ face~$\forest$ is contained in at least three facets~$\tree \eqdef \forest \cup \{(i_\bullet, j_\circ)\}$, $\tree' \eqdef \forest \cup \{(i'_\bullet, j'_\circ)\}$, and~$\tree'' \eqdef \forest \cup \{(i''_\bullet, j''_\circ)\}$. By maximality of~$\tree, \tree', \tree''$, the edges~$(i_\bullet, j_\circ)$, $(i'_\bullet, j'_\bullet)$ and~$(i''_\bullet, j''_\bullet)$ are pairwise crossing and all in the same cell of~$\conv(I_\bullet \cup J_\circ) \ssm \forest$. Therefore,~$i_\bullet, i'_\bullet, i''_\bullet$ are all smaller than~$j_\circ, j'_\circ, j''_\circ$ and we obtain that either~$(i_\bullet, j'_\circ)$ or~$(i_\bullet, j''_\circ)$ (or both) does not belong to~$\tree$ and does not cross any edge of~$\tree$, contradicting the maximality of~$\tree$.
\end{proof}

We say that two $(\signature, I_\bullet, J_\circ)$-trees~$\tree$ and~$\tree'$ are \defn{adjacent}, or related by a \defn{flip}, if they share all but one edge, \ie if there is~$(i_\bullet, j_\circ) \in \tree$ and~$(i'_\bullet, j'_\circ) \in \tree'$ such that~$\tree \ssm \{(i_\bullet, j_\circ)\} = \tree' \ssm \{(i'_\bullet, j'_\circ)\}$. See~\fref{fig:flip}. Note that not all edges of a $(\signature, I_\bullet, J_\circ)$-tree~$\tree$ are flippable: for instance, irrelevant edges of~$\GIJ$ (not crossed by other edges of~$\GIJ$) or leaves of~$\tree$ are never flippable. The following statement characterizes the flippable edges.

\begin{figure}[t]
	\capstart
	\centerline{\includegraphics[scale=.9]{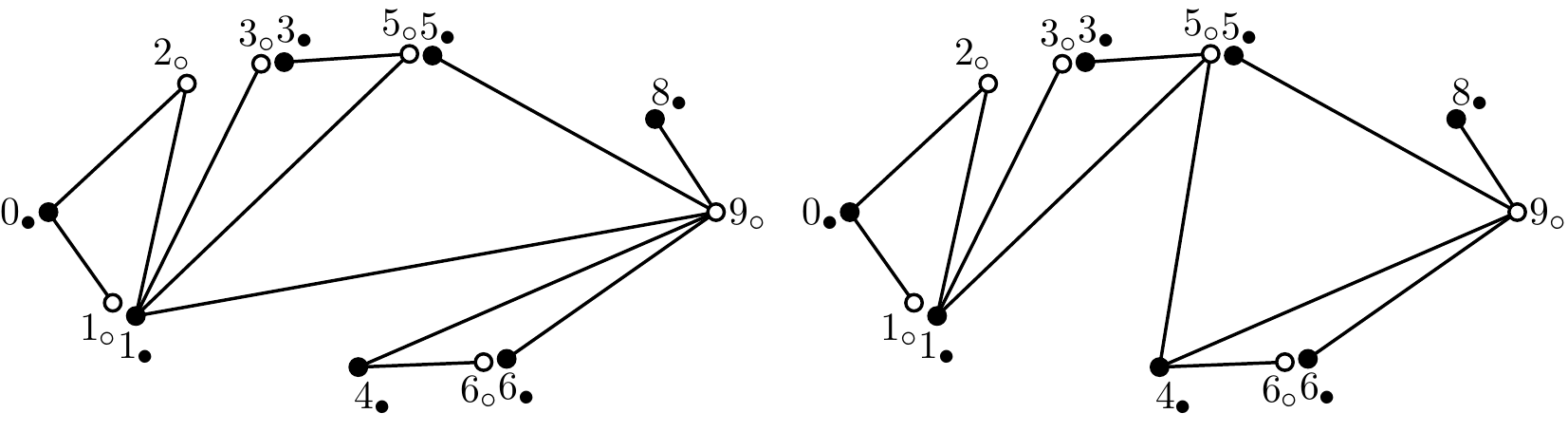}}
	\caption{Two $(\signature, I_\bullet, J_\circ)$-trees related by a flip. Here, ${\signature = {-}{+}{+}{-}{+}{-}{-}{+}}$, ${I_\bullet = \llb 8_\bullet ] \ssm \{2_\bullet, 7_\bullet\}}$ and~$J_\circ = [ 8_\circ \rrb \ssm \{4_\circ, 7_\circ, 8_\circ\}$.}
	\label{fig:flip}
\end{figure}

\begin{proposition}
\label{prop:flippable}
\begin{enumerate}
\item Consider two $(\signature, I_\bullet, J_\circ)$-trees~$\tree$ and~$\tree'$ with~${\tree \ssm \{(i_\bullet, j_\circ)\} = \tree' \ssm \{(i'_\bullet, j'_\circ)\}}$. Then the edges $(i_\bullet, j'_\circ)$ and~$(i'_\bullet, j_\circ)$ are contained in~$\tree$ and~$\tree'$.
\item An edge~$(i_\bullet, j_\circ)$ of a $(\signature, I_\bullet, J_\circ)$-tree~$\tree$ is flippable if and only if there exists~${i'_\bullet \in I_\bullet}$ and~${j'_\circ \in J_\circ}$ such that~$i'_\bullet < j'_\circ$ and both~$(i_\bullet, j'_\circ)$ and~$(i'_\bullet, j_\circ)$ belong to~$\tree$.
\end{enumerate}
\end{proposition}

\begin{proof}
Point~(1) follows by maximality of~$\tree$ since any edge of~$\GIJ$ that crosses~$(i_\bullet, j'_\circ)$ or~$(i'_\bullet, j_\circ)$ also crosses $(i_\bullet, j_\circ)$ or~$(i'_\bullet, j'_\circ)$ (or both). This also shows one direction of Point~(2). For the other direction, we can observe that~$i'_\bullet$ and~$j'_\circ$ are separated by~$(i_\bullet, j_\circ)$ (since the edges~$(i_\bullet, j'_\circ)$ and~$(i'_\bullet, j_\circ)$ are non-crossing) and we assume that~$j_\circ$ and~$j'_\circ$ (resp.~$i_\bullet$ and~$i'_\bullet$) are two consecutive neighbors of~$i_\bullet$ (resp.~of~$j_\circ$) in~$\tree$. The edge~$(i_\bullet, j_\circ)$ can then be flipped to the edge~$(i'_\bullet, j'_\circ)$.
\end{proof}

For instance, the edge~$(4_\bullet, 5_\circ)$ of the $(\signature, I_\bullet, J_\circ)$-tree~$\tree$ of \fref{fig:forestTree}\,(right) can be flipped to~$(1_\bullet, 9_\circ)$ since~$(1_\bullet, 5_\circ)$ and~$(4_\bullet, 9_\circ)$ belong to~$\tree$, see \fref{fig:flip}. In contrast, the edges~$(5_\bullet, 9_\circ)$, $(1_\bullet, 3_\circ)$ and~$(1_\bullet, 5_\circ)$ of~$\tree$ are not flippable: the first is irrelevant, the second is a leaf, the last is neither irrelevant nor a leaf but still does not satisfy the condition of Proposition~\ref{prop:flippable}\,(2).

To conclude, we discuss the boundary of the $(\signature, I_\bullet, J_\circ)$-complex~$\complexIJ$.
The following lemma characterizes the boundary faces of the $(\signature, I_\bullet, J_\circ)$-complex.

\begin{lemma}
\label{lem:interiorFaces}
A $(\signature, I_\bullet, J_\circ)$-forest~$\forest$ lies on the boundary of the $(\signature, I_\bullet, J_\circ)$-complex~$\complexIJ$ if and only if there exists a $(\signature, I_\bullet, J_\circ)$-tree~$\tree$ with an unflippable edge~$\delta$ such that~$\forest \subseteq \tree \ssm \{\delta\}$.
In particular, all $(\signature, I_\bullet, J_\circ)$-forests with a missing irrelevant edge or an isolated node lie on the boundary of~$\complexIJ$.
\end{lemma}

\begin{proof}
By definition, the codimension~$1$ faces on the boundary of~$\complexIJ$ are precisely the faces of the form~$\tree \ssm \{\delta)\}$ where~$\tree$ is a $(\signature, I_\bullet, J_\circ)$-tree and $\delta$ is an unflippable edge of~$\tree$. The first statement thus immediately follows. Finally, any $(\signature, I_\bullet, J_\circ)$-forest~$\forest$ with a missing relevant edge~$\delta$ (resp.~an isolated node~$v$) can be completed into a tree~$\tree$ where~$\delta$ is unflippable (resp.~where~$v$ is a leaf) and~$\forest \subseteq \tree \ssm \{\delta\}$ (resp.~$\forest \subseteq \tree \ssm \{v\}$).
\end{proof}

For instance, consider the $(\signature, I_\bullet, J_\bullet)$-forest~$\forest$ and the $(\signature, I_\bullet, J_\bullet)$-tree~$\tree$ of \fref{fig:forestTree}. The forest~$\forest$ lies on the boundary of~$\complexIJ$ as it can be complete into~$\tree \ssm \{(6_\bullet, 9_\circ)\}$ (the irrelevant edge~$(6_\circ, 9_\bullet)$ is missing), $\tree \ssm \{(1_\bullet, 3_\circ)\}$ (the vertex~$3_\circ$ is isolated) or~$\tree \ssm \{(1_\bullet, 5_\circ)\}$. 
The $(\signature, I_\bullet, J_\circ)$-forests which are not on the boundary of the $(\signature, I_\bullet, J_\circ)$-complex~$\complexIJ$ are called \defn{internal $(\signature, I_\bullet, J_\circ)$-forests}.

\subsection{$\signature$-trees versus triangulations of~$\polygonS$}

We now focus on the situation where~$I_\bullet = \llb n_\bullet ]$ and~$J_\circ = [ n_\circ \rrb$. We write~$\G$ for~$\GIJ[{\llb n_\bullet ]}][{[ n_\circ \rrb}]$ and we just call \defn{$\signature$-trees} (resp.~forests, resp.~complex) the $(\signature, \llb n_\bullet ], [ n_\circ \rrb)$-trees (resp.~forests, resp.~complex). The following immediate bijection between triangulations of~$\polygonS$ and $\signature$-trees is illustrated in \fref{fig:bijectionTreesTriangulations}.

\begin{figure}[b]
	\capstart
	\centerline{\includegraphics[scale=.9]{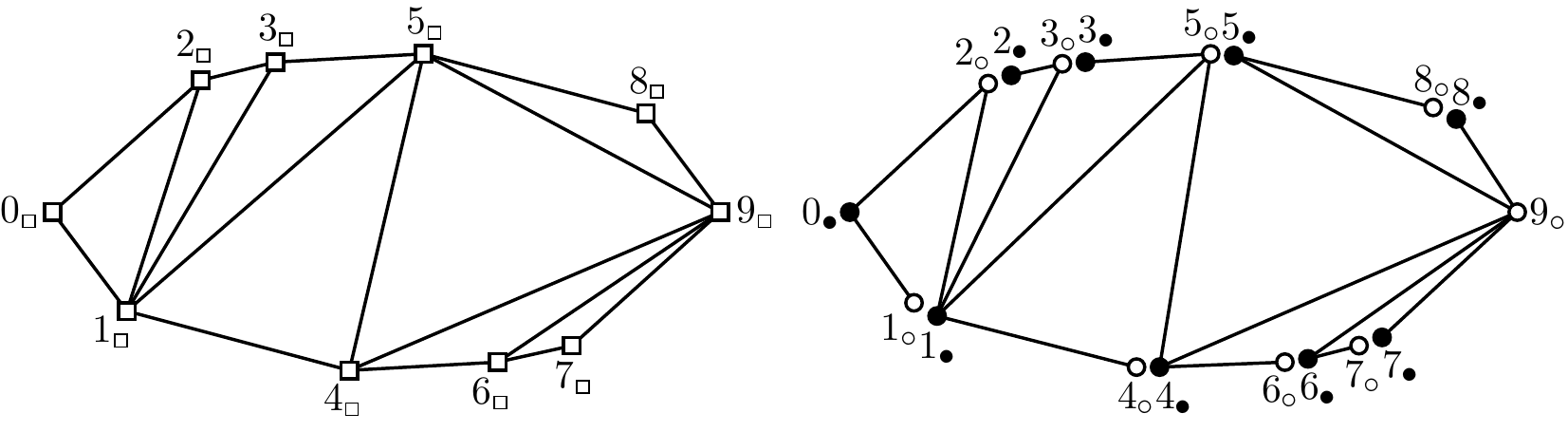}}
	\caption{A triangulation~$T$ of~$\polygonS$ (left) and the corresponding $\signature$-tree $\phi(T)$ (right).}
	\label{fig:bijectionTreesTriangulations}
\end{figure}

\begin{proposition}
\label{prop:bijectionTreesTriangulations}
The map~$\phi$ defined by~$\phi \big( (i_\sq, j_\sq) \big) = (i_\bullet, j_\circ)$ (for~$i_\bullet < j_\circ$) is a bijection between the diagonals of~$\polygonS$ and the edges of~$\G$ and induces a bijection between the dissections (resp.~triangulations) of~$\polygonS$ and the $\signature$-forests (resp.~$\signature$-trees). In particular, the $\signature$-complex is a simplicial associahedron.
\end{proposition}

\begin{proof}
The map $\phi$ is clearly bijective and sends crossing (resp.~non-crossing) diagonals of~$\polygonS$ to crossing (resp.~non-crossing) edges of~$\G$. Therefore, it sends dissections of~$\polygonS$ to $\signature$-forests. Finally, it sends triangulations of $\polygonS$ to $\signature$-trees since a triangulation of~$\polygonS$ has $2n+1$ diagonals (including the boundary edges of~$\polygonS$) and a $\signature$-tree has $2n+1$ edges.
\end{proof}

\begin{corollary}
\label{coro:catalan}
For any signature~$\signature \in \{\pm\}^n$, there are $\cat(n) \eqdef \frac{1}{n+1} \binom{2n}{n}$ many $\signature$-trees.
\end{corollary}

\subsection{Non-crossing matchings}

We conclude this section with another family of non-crossing subgraphs of~$\GIJ$ that will be needed later in the proof of Proposition~\ref{prop:CambrianTriangulations}.
A \defn{perfect matching} of~$\GIJ$ is a subset~$M$ of edges of~$\GIJ$ such that each vertex of~$\GIJ$ is contained in precisely one edge of~$M$. The following statement is immediate.

\begin{lemma}
The bipartite graph~$\GIJ$ admits a perfect matching if and only if~$|I_\bullet| = |J_\circ|$ and $| I_\bullet \cap \llb k_\bullet ] | \ge | J_\circ \cap [ k_\circ \rrb |$ for all~$k \in [n]$. 
\end{lemma}

A matching is \defn{non-crossing} if any two of its edges are non-crossing. See \fref{fig:noncrossingMatching}.

\begin{lemma}
\label{lem:noncrossingmatchings}
If~$\GIJ$ admits a perfect matching, then it has a unique non-crossing perfect matching.
\end{lemma}

\begin{proof}
We give an algorithm to construct the unique non-crossing perfect matching of~$\GIJ$. We consider a vertical pile~$P$ initially empty. We then read the vertices of~$I_\bullet \cup J_\circ$ from left to right. At each step, we read a new vertex~$k$ and proceed as follows:
\begin{itemize}
\item If~$k \in I_\bullet$, we insert~$k$ on top of~$P$ if $\signature_k = {+}$ and at the bottom of~$P$ if~$\signature_k = {-}$.
\item If~$k \in J_\circ$, then we pop the element~$\ell$ on top of~$P$ if $\signature_k = {+}$ and at the bottom of~$P$ if~$\signature_k = {-}$, and connect~$k$ to~$\ell$.
\end{itemize}
This algorithm clearly terminates and returns a non-crossing matching as soon as the pile~$P$ is never empty when an element of~$J_\circ$ is found. This is ensured by the condition~${| I_\bullet \cap \llb k_\bullet ] | \ge | J_\circ \cap [ k_\circ \rrb |}$ for all~$k \in [n]$.
To see that it constructs the unique non-crossing matching, observe that when a vertex~$k \in J_\circ$ is found, we have no other choice than connecting it immediately to the last available vertex on top of~$P$ if $\signature_k = {+}$ and at the bottom of~$P$ if~$\signature_k = {-}$. Indeed, any other choice would separate some vertices of~$P$ to the remaining vertices of~$J_\circ$, and thus ultimately lead to a matching with crossings.
\end{proof}

\begin{remark}
Note that Lemma~\ref{lem:noncrossingmatchings} provides another proof that non-crossing subgraphs of~$\GIJ$ are acyclic. Indeed, since~$\GIJ$ is bipartite, any non-crossing cycle could be decomposed into two distinct non-crossing matchings, contradicting Lemma~\ref{lem:noncrossingmatchings}.
\end{remark}

\begin{figure}[b]
	\capstart
	\centerline{\includegraphics[scale=.9]{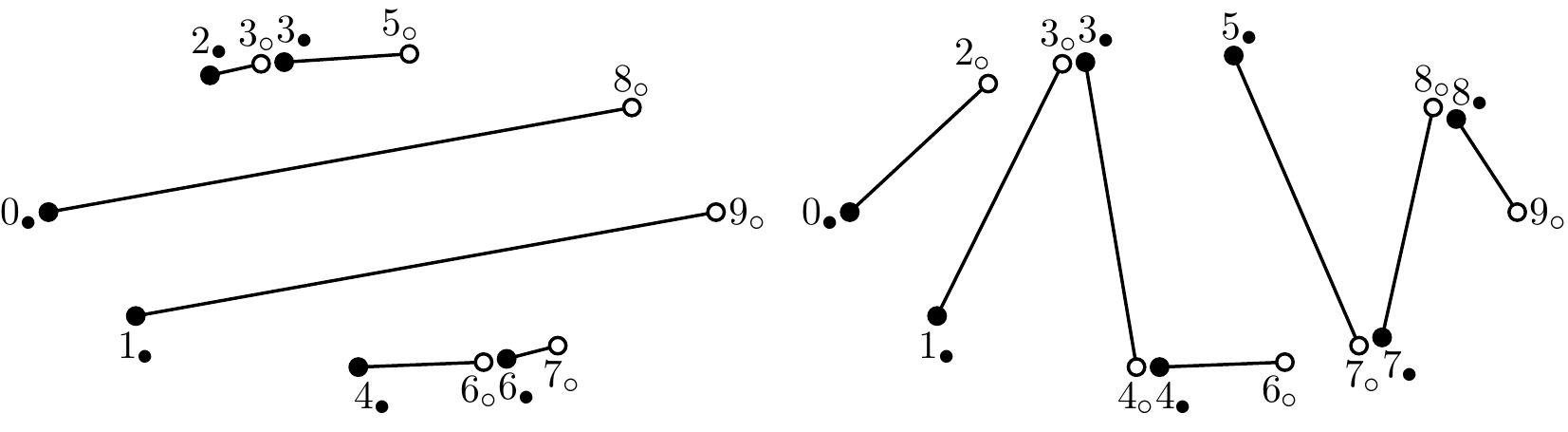}}
	\caption{The unique non-crossing matching of~$\GIJ$ for two distinct instances of~$I_\bullet$ and~$J_\circ$.}
	\label{fig:noncrossingMatching}
\end{figure}

\section{The $(\signature, I_\bullet, J_\circ)$-lattice}

In this section, we orient flips between $(\signature, I_\bullet, J_\circ)$-trees as follows.

\begin{lemma}
\label{lem:characterizationIncreasingFlip}
Consider two adjacent $(\signature, I_\bullet, J_\circ)$-trees~$\tree$ and~$\tree'$ with~$\tree \ssm \{(i_\bullet, j_\circ)\} = \tree' \ssm \{(i'_\bullet, j'_\circ)\}$. We say that the flip from~$\tree$ to~$\tree'$ is \defn{slope increasing} (or simply \defn{increasing}) when the following equivalent conditions hold:
\begin{enumerate}
\item the slope of~$(i_\bullet, j_\circ)$ is smaller than the slope of~$(i'_\bullet, j'_\circ)$,
\item $i'_\bullet$ lies below (resp.~$j'_\circ$ lies above) the line passing through~$i_\bullet$ and~$j_\circ$,
\item the path~$j'_\circ i_\bullet j_\circ i'_\bullet$ in~$\tree$ forms an~$\SSS$ (resp.~the path~$i_\bullet j'_\circ i'_\bullet j_\circ$ in~$\tree'$ forms a~$\ZZZ$).
\end{enumerate}
Otherwise, the flip is called \defn{slope decreasing} (or simply \defn{decreasing}).
\end{lemma}

We leave the immediate proof of this observation to the reader. For example, the flip of \fref{fig:flip} is slope increasing from left to right. In this section, we show that the $(\signature, I_\bullet, J_\circ)$-increasing flip graph is always an interval of the $\signature$-Cambrian lattice of N.~Reading~\cite{Reading-CambrianLattices}.

\subsection{The $(\signature, I_\bullet, J_\circ)$-increasing flip graph}

We call \defn{$(\signature, I_\bullet, J_\circ)$-increasing flip graph}, and denote by~$\IFGIJ$, the oriented graph whose vertices are the $(\signature, I_\bullet, J_\circ)$-trees and whose arcs are increasing flips between them. An example is represented in \fref{fig:lattice}. This section is devoted to some natural properties of this graph, which will be used in the next section to show that the increasing flip graph is the Hasse diagram of a lattice.

We start with some symmetries on $(\signature, I_\bullet, J_\circ)$-increasing flip graphs which will save us later work.
For a signature~$\signature \in \{\pm\}^n$, denote by~$\horimirror{\signature}$ and~$\vertmirror{\signature}$ the signatures of~$\{\pm\}^n$ defined by~$\horimirror{\signature}_k \eqdef -\signature_k$ and~$\vertmirror{\signature}_k \eqdef \signature_{n+1-k}$ for all~$k \in [n]$. For~${I_\bullet \subseteq \llb n_\bullet ]}$ and~$J_\circ \subseteq [ n_\circ \rrb$, define~$\vertmirror{I_\bullet} \eqdef \set{(n+1-i)_\circ}{i_\bullet \in I_\bullet}$ and~$\vertmirror{J_\circ} \eqdef \set{(n+1-j)_\bullet}{j_\circ \in J_\circ}$.

\begin{lemma}
\label{lem:reflections}
The $(\horimirror{\signature}, I_\bullet, J_\circ)$- and $(\vertmirror{\signature}, \vertmirror{J_\circ}, \vertmirror{I_\bullet})$-increasing flip graphs are both isomorphic to the opposite of the $(\signature, I_\bullet, J_\circ)$-increasing flip graph.
\end{lemma}

\begin{proof}
The horizontal and vertical reflections both exchange the flip directions.
\end{proof}

%
%

Let~$\treeMinIJ$ (resp.~$\treeMaxIJ$) denote the set of edges~$\delta$ of~$\GIJ$ such that there is no edge of~$\GIJ$ crossing~$\delta$ with a smaller (resp.~bigger) slope than~$\delta$.
See \fref{fig:minimalMaximal} for an example.

\begin{figure}[b]
	\capstart
	\centerline{\includegraphics[scale=.9]{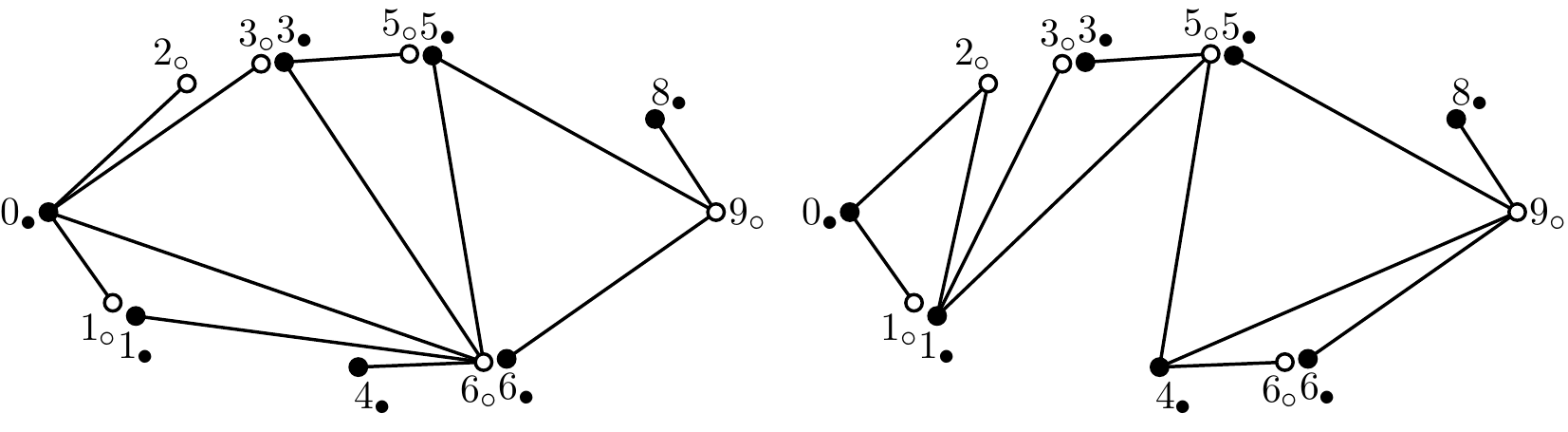}}
	\caption{The minimal (left) and maximal (right) $(\signature, I_\bullet, J_\circ)$-trees. Here, ${\signature = {-}{+}{+}{-}{+}{-}{-}{+}}$, ${I_\bullet = \llb 8_\bullet ] \ssm \{2_\bullet, 7_\bullet\}}$ and~$J_\circ = [ 8_\circ \rrb \ssm \{4_\circ, 7_\circ, 8_\circ\}$.}
	\label{fig:minimalMaximal}
\end{figure}
\begin{figure}
	\capstart
	\centerline{\includegraphics[scale=.57]{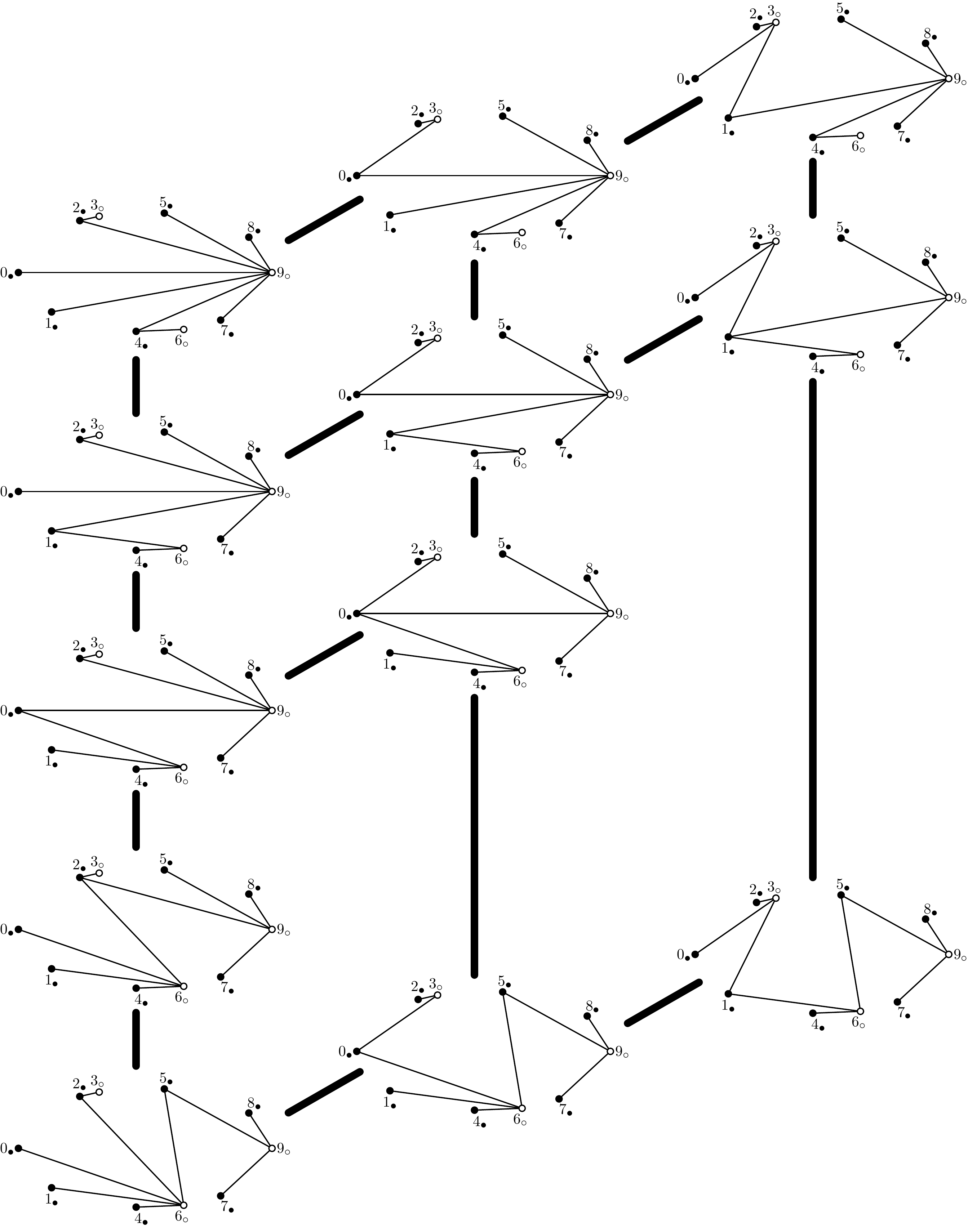}}
	\caption{The $(\signature, I_\bullet, J_\circ)$-lattice on $(\signature, I_\bullet, J_\circ)$-trees. Increasing flips are oriented upwards. Here, ${\signature = {-}{+}{+}{-}{+}{-}{-}{+}}$, ${I_\bullet = \llb 8_\bullet ] \ssm \{3_\bullet, 6_\bullet\}}$ and~$J_\circ = \{3_\circ, 6_\circ, 9_\circ\}$. Compare to Figures~\ref{fig:mixedSubdivision} and~\ref{fig:tropicalHyperplaneArrangement}.}
	\label{fig:lattice}
\end{figure}

\begin{lemma}
The sets~$\treeMinIJ$ and~$\treeMaxIJ$ are $(\signature, I_\bullet, J_\circ)$-trees
\end{lemma}

\begin{proof}
We prove the statement for~$\treeMinIJ$, the statement for~$\treeMaxIJ$ follows by symmetry. The set~$\treeMinIJ$ is clearly non-crossing since among any two crossing edges of~$\GIJ$, only the one of smallest slope can belong to~$\treeMinIJ$. To see that it is inclusion maximal, consider an edge~$(i_\bullet, j_\circ)$ not in~$\treeMinIJ$. Consider the edge~$(i'_\bullet, j'_\circ)$ with the minimal slope among all edges of~$\GIJ$ that cross~$(i_\bullet, j_\circ)$. If~$(i'_\bullet, j'_\circ)$ is not in~$\treeMinIJ$, it is crossed by an edge~$(i''_\bullet, j''_\bullet)$ with smaller slope. Then either~$(i''_\bullet, j''_\circ)$, or~$(i'_\bullet, j''_\circ)$, or~$(i''_\bullet, j'_\circ)$ still crosses~$(i_\bullet, j_\circ)$ and contradicts the minimality of~$(i'_\bullet, j'_\circ)$. We conclude that~$(i_\bullet, j_\circ)$ is crossed by an edge of~$\treeMinIJ$. 
\end{proof}

\begin{proposition}
The $(\signature, I_\bullet, J_\circ)$-increasing flip graph~$\IFGIJ$ is acyclic with a unique source~$\treeMinIJ$ and a unique sink~$\treeMaxIJ$.
\end{proposition}

\begin{proof}
The $(\signature, I_\bullet, J_\circ)$-increasing flip graph~$\IFGIJ$ is clearly acyclic since an increasing flip increases the sum of the slopes of the edges of the $(\signature, I_\bullet, J_\circ)$-tree.

All flips in~$\treeMinIJ$ are increasing by definition, so that~$\treeMinIJ$ is indeed a source. Conversely, any $(\signature, I_\bullet, J_\circ)$-tree~$\tree$ distinct from~$\treeMinIJ$ has a decreasing flip. Indeed, we claim that for any edge~$(i_\bullet, j_\circ) \in \treeMinIJ \ssm \tree$, the edge~$(i'_\bullet, j'_\circ)$ with maximal slope among the edges of~$\tree$ that cross~$(i_\bullet, j_\circ)$ is flippable and its flip is decreasing. To see it, observe first that there exists~${i''_\bullet \in I_\bullet}$ strictly above the line~$(i'_\bullet, j'_\circ)$ such that~$(i''_\bullet, j'_\circ)$ belongs to~$\tree$ and~${i''_\bullet \le \max(i_\bullet, i'_\bullet)}$ . Indeed, take either~$i_\bullet$ or the black endpoint of the edge of~$\tree$ crossing~$(i_\bullet, j'_\circ)$ closest to~$j'_\circ$. Similarly, there exists~$j''_\circ \in J_\circ$ strictly below the line~$(i'_\bullet, j'_\circ)$ such that~$(i'_\bullet, j''_\circ)$ belongs to~$\tree$ and~$j''_\circ \ge \min(j_\circ, j'_\circ)$. Since~$i''_\bullet < j''_\circ$ and~$(i''_\bullet, j'_\circ)$ and~$(i'_\bullet, j''_\circ)$ both belong to~$\tree$, the edge~$(i'_\bullet, j'_\circ)$ is flippable by Proposition~\ref{prop:flippable}\,(2), and since~$i''_\bullet$ is above~$(i'_\bullet, j'_\circ)$ while~$j''_\circ$ is below~$(i'_\bullet, j'_\circ)$, the flip is decreasing by Lemma~\ref{lem:characterizationIncreasingFlip}\,(2).
We conclude that $\treeMinIJ$ is the unique source of the $(\signature, I_\bullet, J_\circ)$-increasing flip graph.
The proof is symmetric for~$\treeMaxIJ$.
\end{proof}

We conclude with a property of the links of the $(\signature, I_\bullet, J_\circ)$-complex.
This property was recently coined \defn{non-revisiting chain property} in~\cite{BarnardMcConville} in the context of graph associahedra.

\begin{proposition}
\label{prop:facialInterval}
The set of $(\signature, I_\bullet, J_\circ)$-trees containing any given $(\signature, I_\bullet, J_\circ)$-forest forms an interval of the $(\signature, I_\bullet, J_\circ)$-increasing flip graph~$\IFGIJ$.
\end{proposition}

\begin{proof}
Consider a $(\signature, I_\bullet, J_\circ)$-forest~$\forest$. Denote by~$C^1, \dots, C^p$ the cells of~$\forest$ (\ie the closures of the connected components of the complement of~$\forest$ in~$\conv(I_\bullet \cup J_\circ)$). For~$k \in [p]$, define~$I_\bullet^k \eqdef I_\bullet \cap C^k$ and~$J_\circ^k \eqdef J_\circ \cap C^k$. Then the subgraph of the $(\signature, I_\bullet, J_\circ)$-increasing flip graph~$\IFGIJ$ induced by the $(\signature, I_\bullet, J_\circ)$-trees containing~$\forest$ is isomorphic to the Cartesian product~$\IFGIJ[I_\bullet^1][J_\circ^1] \times \dots \times \IFGIJ[I_\bullet^p][J_\circ^p]$. We claim that it actually coincides with the interval of the $(\signature, I_\bullet, J_\circ)$-increasing flip graph~$\IFGIJ$ between~$\forest \cup \treeMinIJ[I_\bullet^1][J_\circ^1] \cup \dots \cup \treeMinIJ[I_\bullet^p][J_\circ^p]$ and~$\forest \cup \treeMaxIJ[I_\bullet^1][J_\circ^1] \cup \dots \cup\treeMaxIJ[I_\bullet^p][J_\circ^p]$. For this, we just need to prove that there is no chain of increasing flips that flips out an edge~$\delta$ and later flips back in~$\delta$.

Consider two adjacent $(\signature, I_\bullet, J_\circ)$-trees~$\tree$ and~$\tree'$ with~$\tree \ssm \{(i_\bullet, j_\circ)\} = \tree' \ssm \{(i'_\bullet, j'_\circ)\}$ such that the flip from~$\tree$ to~$\tree'$ is increasing. We claim that any edge~$\delta$ of~$\GIJ$ crossing an edge~$\gamma$ of~$\tree$ with bigger slope also crosses an edge~$\gamma'$ of~$\tree'$ with bigger slope. Indeed, if~$\gamma \ne (i_\bullet, j_\circ)$, then~$\gamma$ still belongs to~$\tree'$ and~$\gamma' = \gamma$ suits. If~$\gamma = (i_\bullet, j_\circ)$, then $\delta \ne (i'_\bullet, j'_\circ)$ since the slope of~$\delta$ is smaller than that of~$\gamma = (i_\bullet, j_\circ)$ in turn smaller than that of~$(i'_\bullet, j'_\circ)$. Therefore, $\delta$ must cross two boundary edges of the square~$i_\bullet i'_\bullet j_\circ j'_\circ$. Since~$\delta$ crosses~$(i_\bullet, j_\circ)$, it thus crosses either~$(i'_\bullet, j'_\circ)$, or~$(i_\circ, j'_\circ)$, or~$(i'_\bullet, j_\circ)$, or the three of them (in which case we choose~$\gamma' = (i'_\bullet, j'_\circ)$). Note that these three edges belong to~$\tree'$ by Proposition~\ref{prop:flippable}\,(1). Moreover, the slope of~$\delta$ is still smaller than the slope of~$\gamma'$.

Consider now a sequence~$\tree_1, \dots, \tree_p$ of $(\signature, I_\bullet, J_\circ)$-trees related by increasing flips. Assume that an edge $\delta$ is flipped out from~$\tree_k$ to~$\tree_{k+1}$. Then~$\delta$ crosses an edge of~$\tree_{k+1}$ with bigger slope, and thus by induction it crosses an edge of~$\tree_\ell$ with bigger slope for any~$\ell > k$. Therefore, $\delta$ cannot be flipped back in by an increasing flip.
\end{proof}

\subsection{The $(\signature, I_\bullet, J_\circ)$-lattice}

The goal of this section is to prove the following statement.

\begin{theorem}
\label{thm:lattice}
The~$(\signature, I_\bullet, J_\circ)$-increasing flip graph~$\IFGIJ$ is the Hasse diagram of a lattice, called \defn{$(\signature, I_\bullet, J_\circ)$-lattice} and denoted by~$\LatIJ$.
\end{theorem}

We start by considering the case when~$I_\bullet = \llb n_\bullet ]$ and~$J_\circ = [ n_\circ \rrb$. Recall that two triangulations~$T$ and~$T'$ of~$\polygonS$ are related by an increasing flip if there exist diagonals~$\delta \in T$ and~$\delta' \in T'$ such that~${T \ssm \{\delta\} = T' \ssm \{\delta'\}}$ and the slope of~$\delta$ is smaller than the slope of~$\delta'$. It is known that the transitive closure of the increasing flip graph is a lattice, called the \defn{$\signature$-Cambrian lattice}~\cite{Reading-CambrianLattices}.

\begin{lemma}
\label{lem:CambrianLattice}
The bijection~$\phi$ or Proposition~\ref{prop:bijectionTreesTriangulations} between triangulations of~$\polygonS$ and $\signature$-trees preserves increasing flips. Therefore, the transitive closure of the increasing flip graph on $\signature$-trees is isomorphic to the $\signature$-Cambrian lattice.
\end{lemma}

In the classical Tamari case when~$\signature = {-}^n$, the $(I_\bullet, J_\circ)$-lattice is isomorphic to the $\nu(I_\bullet, J_\circ)$-Tamari lattice of~\cite{PrevilleRatelleViennot} for some Dyck path $\nu(I_\bullet, J_\circ)$ described in details in~\cite[Sect.~3]{CeballosPadrolSarmiento}. Moreover, it is always an interval of the Tamari lattice. We will prove Theorem~\ref{thm:lattice} via the following generalization of this statement.

\begin{theorem}
\label{thm:interval}
The $(\signature, I_\bullet, J_\circ)$-lattice~$\LatIJ$ is an interval of the $\signature$-Cambrian lattice.
\end{theorem}

In fact, computational experiments indicate the following generalization of Theorem~\ref{thm:interval}.

\begin{conjecture}
\label{conj:interval}
For any~$I_\bullet \subseteq I'_\bullet$ and~$J_\circ \subseteq J'_\circ$, the $(\signature, I_\bullet, J_\circ)$-lattice~$\LatIJ$ is an interval of the $(\signature, I'_\bullet, J'_\circ)$-lattice~$\LatIJ[I'_\bullet][J'_\circ]$.
\end{conjecture}

Although we are not able to prove Conjecture~\ref{conj:interval} in full generality, we will prove Theorem~\ref{thm:interval} using the following three special cases of Conjecture~\ref{conj:interval}.

\begin{lemma}
\label{lem:interval1}
For any~$K \subseteq [n]$, the $(\signature, \llb n_\bullet ] \ssm K_\bullet, [ n_\circ \rrb \ssm K_\circ)$-lattice is an interval of the $\signature$-Cambrian lattice.
\end{lemma}

\begin{proof}
For any vertex~$k \in [n]$, let~$\delta(k)$ be the edge of~$\G$ joining the vertex of~$\llb n_\bullet ]$ preceding~$k$ to the vertex of~$[ n_\circ \rrb \ssm K_\circ$ following~$k$ on the boundary of~$\polygonC$. The ${(\signature, \llb n_\bullet ] \ssm K_\bullet, [ n_\circ \rrb \ssm K_\circ)}$-increasing flip graph is clearly isomorphic to the subgraph of the $(\signature, \llb n_\bullet ], [ n_\circ \rrb)$-increasing flip graph induced by the $(\signature, \llb n_\bullet ], [ n_\circ \rrb)$-trees containing~$\set{\delta(k)}{k \in K}$. It is therefore an interval of the $\signature$-Cambrian lattice by Proposition~\ref{prop:facialInterval} and Lemma~\ref{lem:CambrianLattice}.
\end{proof}

\begin{lemma}
\label{lem:interval2}
For any boundary edge~$(i_\bullet, j_\circ)$ of~$\conv(I_\bullet \cup J_\circ)$ with~$i_\bullet \ne \min(I_\bullet)$ (resp.~${j_\circ \ne \max(J_\circ)}$), the $(\signature, I_\bullet \ssm \{i_\bullet\}, J_\circ)$-lattice (resp.~$(\signature, I_\bullet, J_\circ \ssm \{j_\circ\})$-lattice) is an interval of the $(\signature, I_\bullet, J_\circ)$-lattice.
\end{lemma}

\begin{proof}
By Lemma~\ref{lem:reflections}, we focus on the case where~$j_\circ$ is distinct from~$\max(J_\circ)$ and lies on the lower hull of~$\conv(I_\bullet \cup J_\circ)$.
The $(\signature, I_\bullet, J_\circ \ssm \{j_\circ\})$-increasing flip graph is clearly isomorphic to the subgraph of the $(\signature, I_\bullet, J_\circ)$-increasing flip graph induced by $(\signature, I_\bullet, J_\circ)$-trees with a leaf at~$j_\circ$, or equivalently with an edge~$(i_\bullet, k_\circ)$ with~$k_\circ > j_\circ$. Let~$\ell_\circ$ be the vertex of~$J_\circ$ following~$j_\circ$ along the boundary of~$\conv(I_\bullet \cup J_\circ)$ (which exists since~$j_\circ \ne \max(J_\circ)$). Let~$\treeMin$ denote the minimal $(\signature, I_\bullet, J_\circ)$-tree containing~$(i_\bullet, \ell_\circ)$ (which exists by Proposition~\ref{prop:facialInterval}). We claim that the set of $(\signature, I_\bullet, J_\circ)$-trees containing an edge~$(i_\bullet, k_\circ)$ with~$k_\circ > j_\circ$ is precisely the interval above~$\treeMin$ in the $(\signature, I_\bullet, J_\circ)$-increasing flip graph. We proceed in two steps, showing both inclusions:
\begin{itemize}
\item Observe first that any~$(\signature, I_\bullet, J_\circ)$-tree below~$\treeMin$ contains an edge~$(i_\bullet, k_\circ)$ with~$k_\circ > j_\circ$. Indeed, this property holds for~$\treeMin$ (as it contains the edge~$(i_\bullet, \ell_\circ)$), and it is preserved by a increasing flip (using Proposition~\ref{prop:flippable}\,(1)). 
\item Conversely, consider a $(\signature, I_\bullet, J_\circ)$-tree~$\tree$ containing an edge~$(i_\bullet, k_\circ)$ with~$k_\circ > j_\circ$. Let~$X$ be the half space bounded by~$(i_\bullet, k_\circ)$ containing~$\ell_\circ$, and consider~${\bar I_\bullet \eqdef I_\bullet \cap X}$, ${\bar J_\circ = J_\circ \cap X}$, and~${\bar \tree = \tree \cap X}$. Note that the minimal $(\signature, \bar I_\bullet, \bar J_\circ)$-tree~$\treeMinIJ[\bar I_\bullet][\bar J_\circ]$ contains~$(i_\bullet, \ell_\circ)$. Therefore, using a sequence of decreasing flips from~$\bar \tree$ to~$\treeMinIJ[\bar I_\bullet][\bar J_\circ]$, we can transform~$\tree$ into a \mbox{$(\signature, I_\bullet, J_\circ)$-tree~$\tree'$} containing~$(i_\bullet, \ell_\circ)$. Finally, there is a sequence of decreasing flips from~$\tree'$ to~$\treeMin$ since~$\treeMin$ is the minimal $(\signature, I_\bullet, J_\circ)$-tree containing~$(i_\bullet, \ell_\circ)$. \qedhere
\end{itemize}
\end{proof}

\begin{lemma}
\label{lem:interval3}
For any~$i < j < k$ such that~$(i_\bullet, j_\circ)$ and~$(j_\bullet, k_\circ)$ are boundary edges of~$\conv(I_\bullet \cup J_\circ)$, the $(\signature, I_\bullet \ssm \{i_\bullet\}, J_\circ \ssm \{j_\circ\})$- and $(\signature, I_\bullet \ssm \{j_\bullet\}, J_\circ \ssm \{k_\circ\})$-lattices are intervals of the $(\signature, I_\bullet, J_\circ)$-lattice.
\end{lemma}

\begin{proof}
By Lemma~\ref{lem:reflections}, we can focus on the $(\signature, I_\bullet \ssm \{i_\bullet\}, J_\circ \ssm \{j_\circ\})$-lattice and on the case where~$(i_\bullet, j_\circ)$ and~$(j_\bullet, k_\circ)$ are lower edges of~$\conv({I_\bullet \cup J_\circ})$. The result is also immediate if~$i_\bullet = \min(I_\bullet)$, so we assume otherwise. Let~$E$ be the set of edges of~$\GIJ$ of the form~$(p_\bullet, q_\circ)$ with~$p_\bullet < i_\bullet$ and~$q_\circ > j_\circ$. Let~$(p_\bullet, q_\circ)$ be the edge of maximal slope in~$E$. 

Let~$\cI$ be the interval (by Proposition~\ref{prop:facialInterval}) of the $(\signature, I_\bullet, J_\circ)$-increasing flip graph induced by the $(\signature, I_\bullet, J_\circ)$-trees containing~$(i_\bullet, k_\circ)$. Let~$\treeMax$ be the maximal $(\signature, I_\bullet, J_\circ)$-tree of~$\cI$ containing the edge~$(p_\bullet, q_\circ)$ (which exists by Proposition~\ref{prop:facialInterval}). We claim that the interval below~$\treeMax$ in~$\cI$ is precisely the set of $(\signature, I_\bullet, J_\circ)$-trees containing~$(i_\bullet, k_\circ)$ and an edge of~$E$. The proof of this claim, similar to that of Lemma~\ref{lem:interval2} (showing both inclusions), is left to the reader.

Finally, we observe that there is a bijection~$\psi$ between the $(\signature, I_\bullet \ssm \{i_\bullet\}, J_\circ \ssm \{j_\circ\})$-trees and the $(\signature, I_\bullet, J_\circ)$-trees containing~$(i_\bullet, k_\circ)$ and an edge of~$E$. Namely, a \mbox{$(\signature, I_\bullet \ssm \{i_\bullet\}, J_\circ \ssm \{j_\circ\})$-tree~$\tree$} is sent to the $(\signature, I_\bullet, J_\circ)$-tree~$\phi(\tree)$ obtained from~$\tree$ by replacing each edge of the form~$(j_\bullet, \ell_\circ)$ by the edge~$(i_\bullet, \ell_\circ)$, and finally adding the edges~$(i_\bullet, j_\circ)$ and~$(j_\bullet, k_\circ)$. See \fref{fig:bijectionInterval3}. This bijection clearly preserves increasing flips, which concludes the proof.
\begin{figure}[t]
	\capstart
	\centerline{\includegraphics[scale=.9]{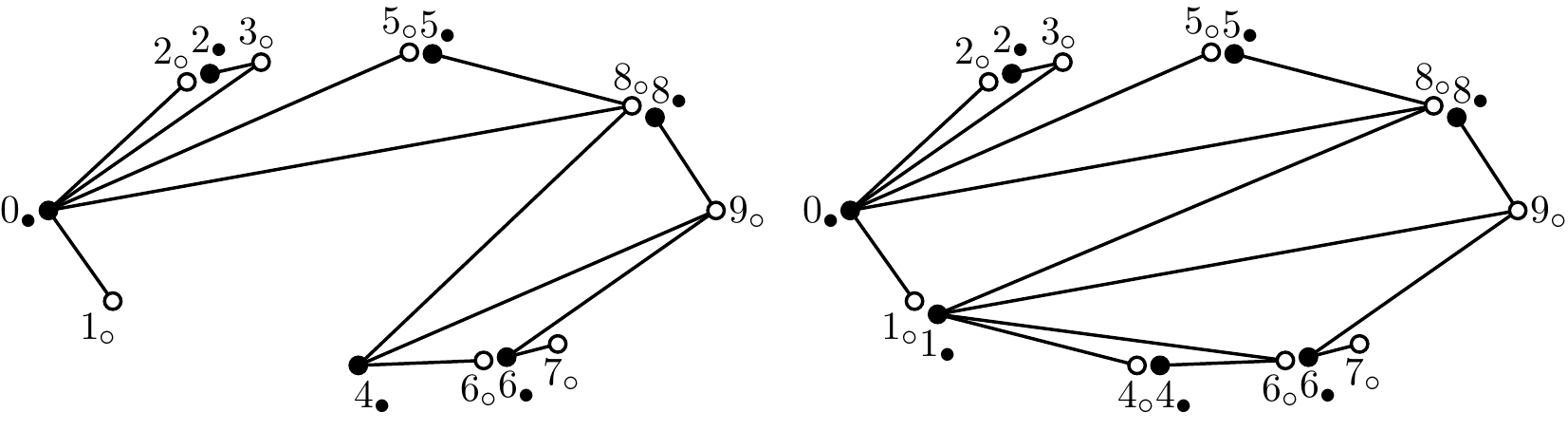}}
	\caption{The bijection of the proof of Lemma~\ref{lem:interval3}. Here, ${\signature = {-}{+}{+}{-}{+}{-}{-}{+}}$, ${I_\bullet = \llb 8_\bullet ] \ssm \{1_\bullet, 7_\bullet\}}$, $J_\circ = [ 8_\circ \rrb \ssm \{4_\circ\}$, $i = 1$, $j = 4$, and~$k = 6$.}
	\label{fig:bijectionInterval3}
\end{figure}
\end{proof}

\begin{proof}[Proof of Theorem~\ref{thm:interval}]
Consider two subsets~$I_\bullet \subseteq \llb n_\bullet ]$ and~$J_\circ \subseteq [ n_\bullet \rrb$ such that~$\min(I_\bullet) < \min(J_\circ)$ and~$\max(I_\bullet) < \max(J_\circ)$. We proceed in three steps:
\begin{enumerate}
\item Let~$I'_\bullet$ be the set of points~$i_\bullet \notin I_\bullet$ such that~$i_\circ$ and the next vertex along~$\conv(I_\bullet \cup J_\circ)$ belong to~$J_\circ$. Similarly, let~$J'_\bullet$ be the set of points~$j_\circ \notin J_\circ$ such that~$j_\bullet$ and the previous vertex along~$\conv(I_\bullet \cup J_\circ)$ belong to~$I_\bullet$. By multiple applications of Lemma~\ref{lem:interval2}, we obtain that the $(\signature, I_\bullet, J_\circ)$-lattice is an interval of the $(\signature, I_\bullet \cup I'_\bullet, J_\circ \cup J'_\circ)$-lattice.
\item After the first step, we can assume that the black and white vertices are alternating along~$\conv(I_\bullet \cup J_\circ)$. Consider now the set~$E$ of edges~$(i_\bullet, j_\circ)$ with~$i_\bullet \notin I_\bullet$ and~$j_\circ \notin J_\circ$, but such that~$(i_\circ, j_\bullet)$ is a boundary edge of~$\conv(I_\bullet \cup J_\circ)$. Note that each such edge~$(i_\bullet, j_\circ)$ is followed a boundary edge~$(j_\bullet, k_\circ)$ of~$\conv(I_\bullet \cup J_\circ)$. Let~$I'_\bullet$ and $J'_\circ$ be the sets of left and right endpoints of the edges of~$E$. By multiple applications of Lemma~\ref{lem:interval3}, we obtain that the $(\signature, I_\bullet, J_\circ)$-lattice is an interval of the $(\signature, I_\bullet \cup I'_\bullet, J_\circ \cup J'_\circ)$-lattice. 
\item After the second step, we can assume that~$k_\bullet \in I_\bullet$ if and only if~$k_\circ \in J_\circ$ for all~$k \in [n]$. A similar argument enables to assume that~$0_\bullet \in I_\bullet$ and~$(n+1)_\circ \in J_\circ$. Therefore, there is~$K \subseteq [n]$ such that~$I_\bullet = \llb n_\bullet ] \ssm K_\bullet$ and~$J_\circ = [ n_\circ \rrb \ssm K_\circ$, and we conclude by Lemma~\ref{lem:interval1} that the $(\signature, I_\bullet, J_\circ)$-lattice is an interval of the $\signature$-Cambrian lattice. \qedhere
\end{enumerate}
\end{proof}

\begin{remark}
Consider a triangulation~$T$ of~$\polygonS$ and its corresponding $\signature$-tree~$\tree \eqdef \phi(T)$. As defined in~\cite{LangePilaud, ChatelPilaud}, the \defn{dual Cambrian tree} of~$T$ (or of~$\tree$) is the (oriented and labeled) tree with
\begin{itemize}
\item one vertex labeled~$j$ for each triangle~$i_\sq j_\sq k_\sq$ of~$T$ with~$i < j < k$,
\item one arc between (the vertices corresponding to) any two adjacent triangles, oriented from the triangle below to the triangle above their common diagonal.
\end{itemize}
The \defn{canopy} of~$T$ (or of~$\tree$) is the sequence of signs~$\canopy(T) = \canopy(\tree) \in \{\pm\}^{n-1}$ defined by~$\canopy(T)_i = {-}$ is~$i$ is below~$i+1$ in the dual Cambrian tree of~$T$, and~$\canopy(T)_i = {+}$ otherwise. The canopy is a natural geometric parameter as it corresponds to the position of the cone of~$T$ in the $\signature$-Cambrian fan of N.~Reading and D.~Speyer~\cite{ReadingSpeyer} with respect to the hyperplanes orthogonal to the simple roots.

For a $\signature$-tree, there is a connection between its canopy and its leaves. Namely, if~$i_\bullet$ is a black leaf of~$\tree$, then $\canopy(\tree)_i \, \signature_i = {+}$ and similarly, if~$j_\circ$ is a white leaf of~$\tree$, then $\canopy(\tree)_{j-1} \, \signature_j = {-}$. When~$\signature = {-}^n$, the reverse implications hold so that the canopy~$\canopy(\tree)$ can be read directly on the tree~$\tree$. In particular, the~$(I_\bullet, J_\circ)$-trees can be identified as the ${-}^n$-trees with particular conditions on their canopy. This enables to derive easily Theorem~\ref{thm:interval} when~$\signature = {-}^n$. However, the reverse implications do not always hold for general signatures. For example, the $\signature$-tree~$\tree$ represented in~\fref{fig:bijectionTreesTriangulations} has~$\canopy(\tree)_5 \, \signature_5 = {+}$ while~$5_\bullet$ is not a leaf of~$\tree$.
\end{remark}

\begin{remark}
It is tempting to attack Conjecture~\ref{conj:interval} by induction on~$|I'_\bullet \ssm I_\bullet| + |J'_\circ \ssm J_\circ|$. By Lemma~\ref{lem:reflections}, it would be sufficient to prove that for any~$i_\bullet \notin I_\bullet$, the $(\signature, I_\bullet, J_\circ)$-lattice is an interval of the $(\signature, I_\bullet \cup \{i_\bullet\}, J_\circ)$-lattice. Lemma~\ref{lem:interval2} shows this fact in the case when the vertex following~$i_\bullet$ along the boundary of~$\conv(I_\bullet \cup J_\circ \cup \{i_\bullet\})$ is in~$J_\circ$. Lemma~\ref{lem:interval2} treats the case when the two vertices following~$i_\bullet$ along the boundary of~$\conv(I_\bullet \cup J_\circ \cup \{i_\bullet\})$ are in~$I_\bullet$ and~$J_\circ$ respectively. However, we did not manage to prove this fact when~$i_\bullet$ is followed by two or more vertices of~$I_\bullet$ along the boundary of~$\conv(I_\bullet \cup J_\circ \cup \{i_\bullet\})$. Note however that we can prove in any case that the $(\signature, I_\bullet, J_\circ)$-lattice is a lattice quotient of an interval of the $(\signature, I_\bullet \cup \{i_\bullet\}, J_\circ)$-lattice.
\end{remark}

We conclude this section with a geometric consequence of Theorem~\ref{thm:interval}. Remember that the $\signature$-Cambrian lattice can be realized geometrically as
\begin{itemize}
\item the dual graph of the $\signature$-Cambrian fan of N.~Reading and D.~Speyer~\cite{ReadingSpeyer},
\item the graph of the $\signature$-associahedron of C.~Hohlweg and C.~Lange~\cite{HohlwegLange}.
\end{itemize}
As an interval of the $\signature$-Cambrian lattice gives rise to a connected region of the $\signature$-Cambrian fan, we obtain the following geometric realization of the $(\signature, I_\bullet, J_\circ)$-lattice.

\begin{corollary}
\label{coro:regionCambrianFan}
The $(\signature, I_\bullet, J_\circ)$-increasing flip graph is realized geometrically as the dual graph of a set cones of the $\signature$-Cambrian fan of~\cite{ReadingSpeyer} corresponding to an interval of the $\signature$-Cambrian lattice.
\end{corollary}

\section{The $(\signature, I_\bullet, J_\circ)$-triangulation}

In this section, we use $\signature$-trees to construct a flag regular triangulation~$\CambTriang$ of the subpolytope~$\conv \bigset{ \big( \b{e}_{i_\bullet}, \b{e}_{j_\circ} \big)}{0 \le i_\bullet < j_\circ \le n+1}$ of the product of simplices~$\triangle_{\llb n_\bullet ]} \times \triangle_{[ n_\circ \rrb}$. Restricting~$\CambTriang$ to the face~$\triangle_{I_\bullet} \times \triangle_{J_\circ}$ then yields a triangulation whose dual graph is the flip graph on $(\signature, I_\bullet, J_\circ)$-trees.

\subsection{The $\signature$-triangulation}

Let~$(\b{e}_{i_\bullet})_{i_\bullet \in I_\bullet}$ denote the standard basis of~$\R^{I_\bullet}$ and~$(\b{e}_{j_\circ})_{j_\circ \in J_\circ}$ denote the standard basis of~$\R^{J_\circ}$. We consider the Cartesian product of the two standard simplices
\[
\triangle_{I_\bullet} \times \triangle_{J_\circ} \eqdef \conv \bigset{ \big( \b{e}_{i_\bullet}, \b{e}_{j_\circ} \big)}{i_\bullet \in I_\bullet, \, j_\circ \in J_\circ}
\]
and its subpolytope
\[
\subpolyIJ \eqdef \conv \bigset{ \big( \b{e}_{i_\bullet}, \b{e}_{j_\circ} \big)}{i_\bullet \in I_\bullet, \, j_\circ \in J_\circ \text{ and } i_\bullet < j_\circ}.
\]
Note that the polytopes~$\triangle_{I_\bullet} \times \triangle_{J_\circ}$ and $\subpolyIJ$ are faces of the polytopes~$\triangle_{\llb n_\bullet ]} \times \triangle_{[ n_\circ \rrb}$ and~$\subpolyIJ[{\llb n_\bullet ]}][{[ n_\circ \rrb}]$ respectively.

\begin{proposition}
\label{prop:CambrianTriangulations}
Each $(\signature, I_\bullet, J_\circ)$-tree defines a simplex
\(
\triangle_\tree \eqdef \conv \bigset{\big( \b{e}_{i_\bullet}, \b{e}_{j_\circ} \big)}{(i_\bullet, j_\circ) \in \tree}
\)
and the collection of simplices
\(
\CambTriangIJ \eqdef \bigset{\triangle_\tree}{\tree \text{ $(\signature, I_\bullet, J_\circ)$-tree}}
\)
is a flag triangulation of~$\subpolyIJ$, that we call the \defn{$\signature$-triangulation} of~$\subpolyIJ$.
\end{proposition}

%
%

\begin{proof}
Since a triangulation of a polytope induces a triangulation on all its faces, we only need to prove the result for~$I_\bullet = \llb n_\bullet ]$ and~$J_\circ = [ n_\circ \rrb$. Let~$\subpoly \eqdef \subpolyIJ[{\llb n_\bullet ]}][{[ n_\circ \rrb}]$. Observe that:
\begin{itemize}
\item Each~$\triangle_\tree$ is a full-dimensional simplex since~$\tree$ is a spanning tree of~$\G$.
\item For~$\tree \ne \tree'$, the simplices~$\triangle_\tree$ and~$\triangle_{\tree'}$ intersect along a face of both. Otherwise, $\G$ would contain a cycle~$C$ that alternates between~$\tree$ and~$\tree'$, thus providing two distinct non-crossing matchings on the support of~$C$, contradicting Lemma~\ref{lem:noncrossingmatchings}.
\item The total volume of these simplices is the volume of~$\subpoly$. On the one hand, since each simplex is unimodular, Corollary~\ref{coro:catalan} shows that the total normalized volume of the simplices is~$\cat(n)$. On the other hand, the normalized volume of~$\subpoly$ is known to be~$\cat(n)$ as it is triangulated by the (bottom part of the) staircase triangulation~\cite[Sect.~6.2.3]{DeLoeraRambauSantos}.
\end{itemize}
This proves that~$\set{\triangle_\tree}{\tree \text{ $\signature$-tree}}$ is a triangulation of~$\subpoly$. It is clearly flag by definition of $\signature$-trees.
\end{proof}

\begin{remark}
Since the $\signature$-triangulation only depends on the crossings among the edges of~$\G$, the $\horimirror{\signature}$-triangulation coincides with the $\signature$-triangulation, while the $\vertmirror{\signature}$-triangulation is the image of the $\signature$-triangulation by the symmetry that simultaneously exchanges~$\b{e}_{k_\bullet}$ with~$\b{e}_{k_\circ}$ for all~$k \in [n]$.
\end{remark}

\begin{example}
Consider the case~$n = 3$.
Denote the vertices of~$\subpoly$ by
\[
\begin{array}{ccccc}
\alpha = (\b{e}_{0_\bullet}, \b{e}_{1_\circ}) &
\beta = (\b{e}_{0_\bullet}, \b{e}_{2_\circ}) &
\gamma = (\b{e}_{0_\bullet}, \b{e}_{3_\circ}) &
\delta = (\b{e}_{0_\bullet}, \b{e}_{4_\circ}) \\ &
\epsilon = (\b{e}_{1_\bullet}, \b{e}_{2_\circ}) &
\eta = (\b{e}_{1_\bullet}, \b{e}_{3_\circ}) &
\kappa = (\b{e}_{1_\bullet}, \b{e}_{4_\circ}) \\ & &
\lambda = (\b{e}_{2_\bullet}, \b{e}_{3_\circ}) &
\mu = (\b{e}_{2_\bullet}, \b{e}_{4_\circ}) \\ & & &
\nu = (\b{e}_{3_\bullet}, \b{e}_{4_\circ})
\end{array}
\]
Then the $4$ $\signature$-triangulations and the staircase triangulation of~$\subpoly$ are given by the simplices
\[
\begin{array}{c|c|c|c|c||c}
\signature & {-}{-}{-} \text{ or } {+}{+}{+} & {-}{-}{+} \text{ or } {+}{+}{-} & {-}{+}{-} \text{ or } {+}{-}{+} & {-}{+}{+} \text{ or } {+}{-}{-} & \text{staircase triang.}\\
\hline
& \alpha \beta \gamma \delta \epsilon \lambda \nu & \alpha \beta \gamma \delta \epsilon \mu \nu & \alpha \beta \gamma \delta \eta \mu \nu & \alpha \beta \gamma \delta \kappa \lambda \nu & \alpha \beta \gamma \delta \kappa \mu \nu \\
& \alpha \beta \delta \epsilon \lambda \mu \nu & \alpha \beta \gamma \epsilon \lambda \mu \nu & \alpha \beta \gamma \eta \lambda \mu \nu & \alpha \beta \gamma \eta \kappa \lambda \nu & \alpha \beta \gamma \eta \kappa \mu \nu \\
\text{simplices} & \alpha \gamma \delta \epsilon \eta \lambda \nu & \alpha \gamma \delta \epsilon \kappa \mu \nu & \alpha \beta \delta \eta \kappa \mu \nu & \alpha \beta \delta \kappa \lambda \mu \nu & \alpha \beta \gamma \eta \lambda \mu \nu \\
& \alpha \delta \epsilon \eta \kappa \lambda \nu & \alpha \gamma \epsilon \eta \kappa \mu \nu & \alpha \beta \epsilon \eta \kappa \mu \nu & \alpha \beta \epsilon \eta \kappa \lambda \nu & \alpha \beta \epsilon \eta \kappa \mu \nu \\
& \alpha \delta \epsilon \kappa \lambda \mu \nu & \alpha \gamma \epsilon \eta \lambda \mu \nu & \alpha \beta \epsilon \eta \lambda \mu \nu & \alpha \beta \epsilon \kappa \lambda \mu \nu & \alpha \beta \epsilon \eta \lambda \mu \nu
\end{array}
\]
\smallskip
\end{example}

\begin{remark}
Note that Proposition~\ref{prop:CambrianTriangulations} provides an alternative proof of Proposition~\ref{prop:pseudomanifold}.
\end{remark}

\begin{remark}
As a corollary of Proposition~\ref{prop:CambrianTriangulations} and the unimodularity of~$\subpolyIJ$, we obtain that the number of $(\signature, I_\bullet, J_\circ)$-trees is independent of~$\signature$. It is clear for~$I_\bullet = \llb n_\bullet ]$ and~$J_\circ = [ n_\circ \rrb$ using the bijection of Proposition~\ref{prop:bijectionTreesTriangulations} but we have not found a clear combinatorial reason for general~$I_\bullet$ and~$J_\circ$.
\end{remark}

Finally, we gather two geometric consequences of Proposition~\ref{prop:CambrianTriangulations}. The first is just a reformulation of Proposition~\ref{prop:CambrianTriangulations}.

\begin{corollary}
\label{coro:dualTriangulation}
The $(\signature, I_\bullet, J_\circ)$-increasing flip graph is geometrically realized as the dual graph of the triangulation~$\CambTriangIJ$.
\end{corollary}

The second is an application of the Cayley trick~\cite{HuberRambauSantos} to visualize triangulations of products of simplices as mixed subdivisions of generalized permutahedra. Recall that a \defn{generalized permutahedron}~\cite{Postnikov, PostnikovReinerWilliams} is a polytope whose normal fan coarsens the normal fan of the permutahedron~$\conv\set{(\sigma_1, \dots, \sigma_n)}{\sigma \in \fS_n}$. For example, the Minkowski sum~$\sum_{I \subseteq [n]} y_I \, \triangle_I$ is a generalized permutahedron for any family~$(y_I)_{I \subseteq [n]}$ (where~$\triangle_I \eqdef \conv\set{\b{e}_i}{i \in I}$ denotes the face of the standard simplex corresponding to~$I$). The following statement is illustrated in \fref{fig:mixedSubdivision}.

\begin{corollary}
\label{coro:mixedSubdivision}
The collection of generalized permutahedra~$\sum_{i_\bullet \in I_\bullet} \triangle_{\set{j_\circ \in J_\circ}{(i_\bullet, j_\circ) \in \tree}}$, where~$\tree$ ranges over all $(\signature, I_\bullet, J_\circ)$-trees, forms a coherent fine mixed subdivision of the generalized permutahedron~$\sum_{i_\bullet \in I_\bullet} \triangle_{\set{j_\circ \in J_\circ}{i_\bullet < j_\circ}}$.
\end{corollary}

\begin{figure}
	\capstart
	\centerline{\includegraphics[scale=1]{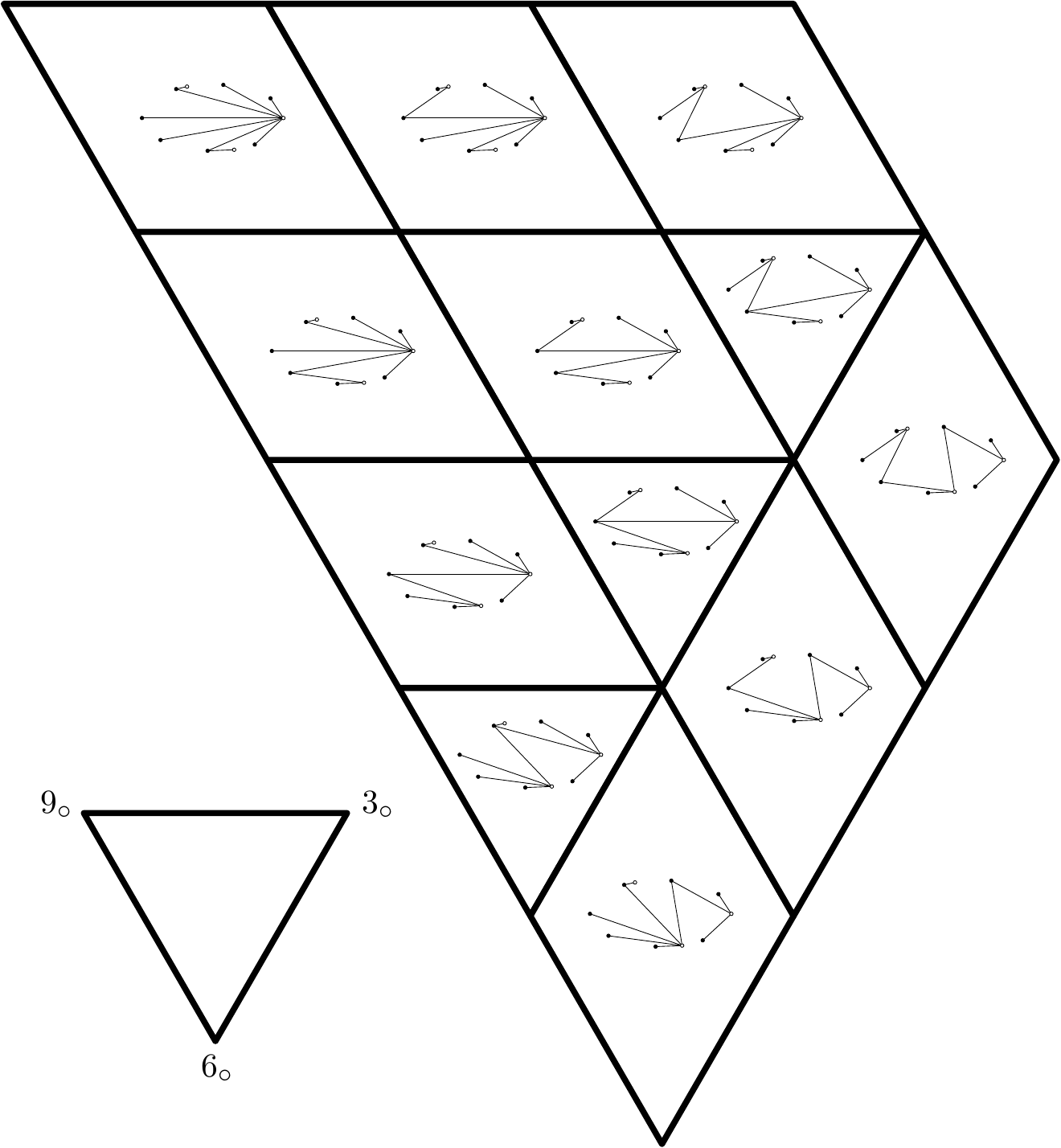}}
	\caption{The mixed subdivision realization of the $(\signature, I_\bullet, J_\circ)$-lattice. Here, ${\signature = {-}{+}{+}{-}{+}{-}{-}{+}}$, ${I_\bullet = \llb 8_\bullet ] \ssm \{3_\bullet, 6_\bullet\}}$ and~$J_\circ = \{3_\circ, 6_\circ, 9_\circ\}$. Compare to Figures~\ref{fig:lattice} and~\ref{fig:tropicalHyperplaneArrangement}.}
	\label{fig:mixedSubdivision}
\end{figure}

\subsection{Regularity}

Recall that a triangulation~$T$ of a point set~$P$ is \defn{regular} if there exists a lifting function~$h : P \to \R$ such that~$T$ is the projection of the lower convex hull of the lifted point set~$\set{(p,h(p))}{p \in P}$.

\begin{proposition}
\label{prop:regular}
For any~$\signature \in \{\pm\}^n$, $I_\bullet \subseteq \llb n_\bullet ]$ and~$J_\circ \subseteq [ n_\circ \rrb$, the triangulation~$\CambTriangIJ$ is regular.
\end{proposition}

\begin{proof}
Consider two adjacent $\signature$-trees~$\tree, \tree'$ with~$\tree \ssm \{(i_\bullet, j_\circ)\} = \tree' \ssm \{(i'_\bullet, j'_\circ)\}$. Then the linear dependence between the vertices of~$\triangle_\tree$ and~$\triangle_\tree'$ is given by
\[
(\b{e}_{i_\bullet}, \b{e}_{j_\circ}) + (\b{e}_{i'_\bullet}, \b{e}_{j'_\circ}) = (\b{e}_{i_\bullet}, \b{e}_{j'_\circ}) + (\b{e}_{i'_\bullet}, \b{e}_{j_\circ}).
\]
Therefore, we just need to find a lifting function~$h : \set{(i_\bullet, j_\circ)}{i_\bullet \in I_\bullet, \, j_\circ \in J_\circ, \, i_\bullet < j_\circ} \to \R$ such that for any two crossing edges~$(i_\bullet, j_\circ)$ and~$(i'_\bullet, j'_\circ)$ of~$\G$, we have
\[
h \big( (i_\bullet, j_\circ) \big) + h \big( (i'_\bullet, j'_\circ) \big) > h \big( (i_\bullet, j'_\circ) \big) + h \big( (i'_\bullet, j_\circ) \big).
\]
For this, consider any strictly concave increasing function~$f: \R \to \R$. For a diagonal~$\zeta$ of~$\polygonS$, we denote by~$\ell(\zeta)$ the minimum between the number of vertices of~$\polygonS$ on each side of the diagonal~$\zeta$. Consider two crossing diagonals~$\zeta$ and~$\eta$ of~$\polygonS$. These diagonals decompose the polygon~$\polygonS$ into four regions that we denote by~$A,B,C,D$ such that~$\zeta$ separates~$A \cup B$ from $C \cup D$ and~$|A \cup B| \le |C \cup D|$, while~$\eta$ separates~$A \cup C$ from $B \cup D$ and~$|A \cup C| \le |B \cup D|$. We also denote accordingly by~$\alpha, \beta, \gamma, \delta$ the boundary edges of the square with diagonals~$\zeta, \eta$. Thus, we have
\begin{gather*}
\ell(\zeta) = |A| + |B| + 1 \qquad\text{and}\qquad \ell(\eta) = |A| + |C| + 1 \\
\text{while}\qquad \ell(\alpha) = |A|, \qquad \ell(\beta) \le |B|, \qquad \ell(\gamma) \le |C|, \qquad\text{and}\qquad \ell(\delta) \le |A| + |B| + |C| + 2.
\end{gather*}
Using the strict concavity of~$f$ for the first inequality and the increasingness for the second inequality, we obtain that
\[
f(\ell(\zeta)) + f(\ell(\eta)) > f(\ell(\alpha)) + f(\ell(\delta))
\qquad\text{and}\qquad
f(\ell(\zeta)) + f(\ell(\eta)) > f(\ell(\beta)) + f(\ell(\gamma)).
\]
Finally, we transport this convenient function through the bijection~$\phi$ of Proposition~\ref{prop:bijectionTreesTriangulations} to obtain a suitable lifting function~$h \eqdef f \circ \ell \circ \phi^{-1}$.
\end{proof}

\begin{remark}
In the classical Tamari case when~$\signature = {-}^n$, the function~$\ell$ in the proof of Proposition~\ref{prop:regular} can be replaced by~${\ell' \big( (i_\bullet, j_\circ) \big) = j_\circ - i_\bullet}$. Note however that this simple function~$\ell'$ fails for arbitrary signatures~$\signature \in \{\pm\}^n$.
\end{remark}

\begin{remark}
Propositions~\ref{prop:CambrianTriangulations} and~\ref{prop:regular} enable us to understand~$2^{n-1}$ distinct regular triangulations of~$\subpoly \eqdef \subpolyIJ[{\llb n_\bullet ]}][{[n_\circ \rrb}]$. It would be interesting to investigate if one can understand similarly more (regular) triangulations of~$\subpoly$. Note that not all regular triangulations of~$\subpoly$ are flag. Some computations:

\bigskip
\centerline{
	\begin{tabular}{l|ccccc}
		$n$ & 1 & 2 & 3 & 4 & 5 \\
		\hline
		\# $\signature$-triangulations of~$\subpoly$ & 1 & 1 & 2 & 4 & 8 \\
		\# regular triangulations of~$\subpoly$ & 1 & 1 & 2 & 20 & 3324 \\
		\# flag regular triangulations of~$\subpoly$ & 1 & 1 & 2 & 16 & 848
	\end{tabular}
}
\medskip
\end{remark}

\section{Tropical realization}

In this section, we exploit the triangulation~$\CambTriang$ to obtain a geometric realization of the $(\signature, I_\bullet, J_\circ)$-lattice as the edge graph of a polyhedral complex induced by a tropical hyperplane arrangement. We follow the same lines as~\cite{CeballosPadrolSarmiento}, relying on work of M.~Develin and B.~Sturmfels~\cite{DevelinSturmfels}.
%
Define the following geometric objects in the tropical projective space~$\TP^{|J_\circ|-1} = \R^{J_\circ} / \R\one$:
\begin{enumerate}
\item For each~$i_\bullet \in I_\bullet$, consider the inverted tropical hyperplane at~$\big( h(i_\bullet, j_\circ) \big)_{j_\circ \in J_\circ}$ defined by
\[
H_{i_\bullet} \eqdef \set{x \in \TP^{|J_\circ|-1}}{\max\nolimits_{j_\circ \in J_\circ} \big\{ x_{j_\circ} - h(i_\bullet, j_\circ) \big\} \text{ is attained twice}}.
\]

\item For each edge~$(i_\bullet, j_\circ)$ of~$\GIJ$, consider the polyhedron
\[
g(i_\bullet, j_\circ) \eqdef \set{x \in \R^{J_\circ}}{x_{k_\circ} - x_{j_\circ} \le h(i_\bullet, k_\circ) - h(i_\bullet, j_\circ) \text{ for each } k_\circ \in J_\circ} \cap \{x_{\max(J_\circ)} = 0\}.
\]

\item For each covering~$(\signature, I_\bullet, J_\circ)$-forest~$\forest$, consider the polyhedron
\[
g(\forest) \eqdef \bigcap_{(i_\bullet, j_\circ) \in \forest} g(i_\bullet, J_\circ).
\]

\item For each $(I_\bullet, J_\circ)$-tree~$\tree$, consider the point~$g(\tree) \in \R^{J_\circ}$ whose~$k_\circ$ coordinate is given by
\[
g(\tree)_{k_\circ} = \sum_{(i_\bullet, j_\circ) \in p(\tree, k_\circ)} \pm h(i_\bullet, j_\circ),
\]
where~$p(\tree, k_\circ)$ is the unique path in~$\tree$ from~$k_\circ$ to~$\max(J_\circ)$, and the sign of the summand~$h(i_\bullet, j_\circ)$ is negative if~$p(\tree, k_\circ)$ traverses~$(i_\bullet, j_\circ)$ from~$i_\bullet$ to~$j_\circ$ and positive otherwise.
\end{enumerate}
We call \defn{$(\signature, I_\bullet, J_\circ)$-associahedron} the polyhedral complex~$\Asso$ given by the bounded cells of the arrangement of tropical hyperplanes~$H_{i_\bullet}$ for~$i_\bullet \in I_\bullet$.
The following statement is identical to that of~\cite{CeballosPadrolSarmiento} and its proof is similar.

\begin{theorem}
\label{thm:tropical}
The $(\signature, I_\bullet, J_\circ)$-associahedron~$\Asso$ is a polyhedral complex whose cell poset is anti-isomorphic to the inclusion poset of interior faces of the $(\signature, I_\bullet, J_\circ)$-complex. In particular,
\begin{itemize}
\item each internal~$(\signature, I_\bullet, J_\circ)$-forest~$\forest$ corresponds to a face~$g(\forest)$ of~$\Asso$;
\item each $(\signature, I_\bullet, J_\circ)$-tree~$\tree$ corresponds to a vertex~$g(\tree)$ of~$\Asso$;
\item each flip corresponds to an edge of~$\Asso$.
\end{itemize}
In particular, the edge graph of~$\Asso$ is the flip graph on $(\signature, I_\bullet, J_\circ)$-trees. In fact, when oriented in the linear direction~$\sum_{j_\circ \in J_\circ \ssm \{\max(J_\circ)\}} \signature_j \, x_{j_\circ}$, the edge graph of~$\Asso$ is the increasing flip graph~$\IFGIJ$ on $(\signature, I_\bullet, J_\circ)$-trees.
\end{theorem}

\begin{example}
Consider~$\signature = {-}{+}{+}{-}{+}{-}{-}{+}$, $I_\bullet = \llb 8_\bullet ] \ssm \{3_\bullet, 6_\bullet\}$ and~$J_\circ = \{3_\circ, 6_\circ, 9_\circ\}$. We consider the lifting function~$h(i_\bullet, j_\circ) = \sqrt{\ell(i_\sq, j_\sq)}$ which gives
\[
h(i_\bullet, j_\circ) =
\begin{blockarray}{cccccccc}
	0_\bullet & 1_\bullet & 2_\bullet & 4_\bullet & 5_\bullet & 7_\bullet & 8_\bullet & \\
	\begin{block}{(ccccccc)c}
	1 & \sqrt{2} & 0 & \infty & \infty & \infty & \infty & 3_\circ \\
	\sqrt{2} & 1 & \sqrt{3} & 0 & \sqrt{3} & \infty & \infty & 6_\circ \\
	2 & \sqrt{3} & \sqrt{3} & \sqrt{2} & 1 & 0 & 0 & 9_\circ \\
	\end{block}
\end{blockarray}
\]
The corresponding tropical hyperplane arrangement is represented in \fref{fig:tropicalHyperplaneArrangement}.
Oriented north-east, it coincides with the increasing flip graph represented in \fref{fig:lattice}.

\begin{figure}[h]
	\capstart
	\centerline{\includegraphics[scale=1]{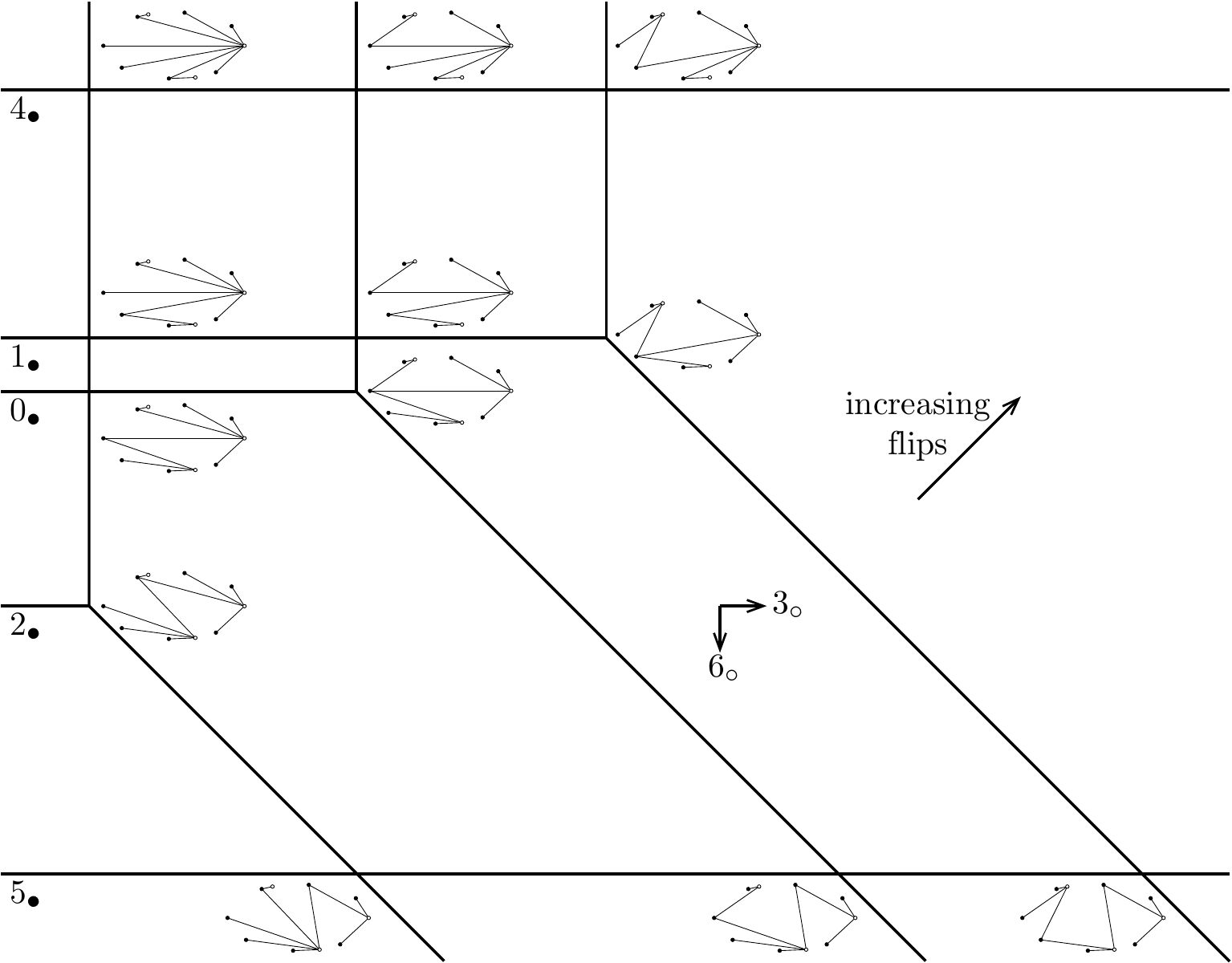}}
	\caption{The tropical realization of the $(\signature, I_\bullet, J_\circ)$-lattice. Here, ${\signature = {-}{+}{+}{-}{+}{-}{-}{+}}$, ${I_\bullet = \llb 8_\bullet ] \ssm \{3_\bullet, 6_\bullet\}}$ and~$J_\circ = \{3_\circ, 6_\circ, 9_\circ\}$. Compare to Figures~\ref{fig:lattice} and~\ref{fig:mixedSubdivision}. Note that~$H_{4_\bullet}$ and~$H_{5_\bullet}$ are degenerate tropical hyperplanes and that~$H_{7_\bullet}$ is at infinity.}
	\label{fig:tropicalHyperplaneArrangement}
\end{figure}

We have computed some coordinates of $(\signature, I_\bullet, J_\circ)$-trees in \fref{fig:coordinates}. For example,
\begin{align*}
g(\tree_{\min})_{3_\circ} & = -h(2_\bullet, 6_\circ) + h(5_\bullet, 6_\circ) - h(5_\bullet, 9_\circ) = -1, \\
\text{and}\qquad
g(\tree_{\min})_{6_\circ} & = h(5_\bullet, 6_\circ) - h(5_\bullet, 9_\circ) = \sqrt{3}-1.
\end{align*}

\begin{figure}[h]
	\capstart
	\centerline{
	\begin{tabular}{c@{\qquad}c@{\qquad}c}
	\includegraphics[scale=.57]{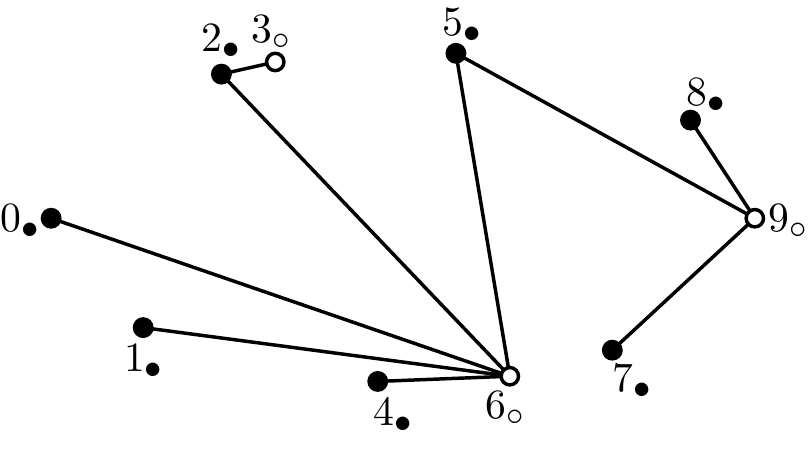} & \includegraphics[scale=.57]{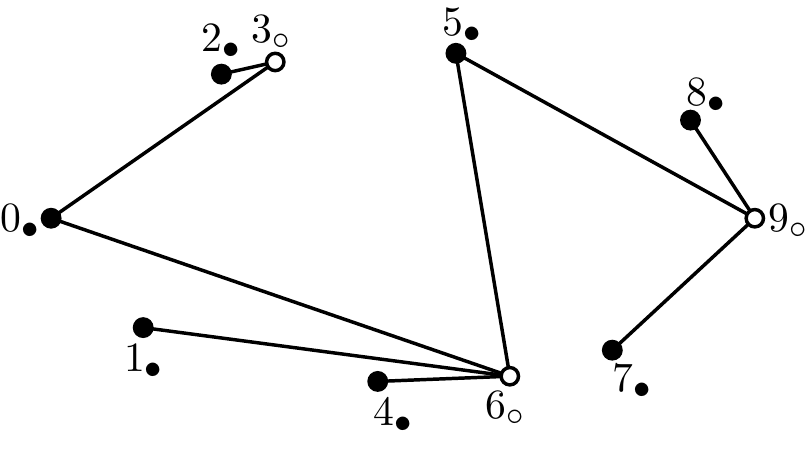} & \includegraphics[scale=.57]{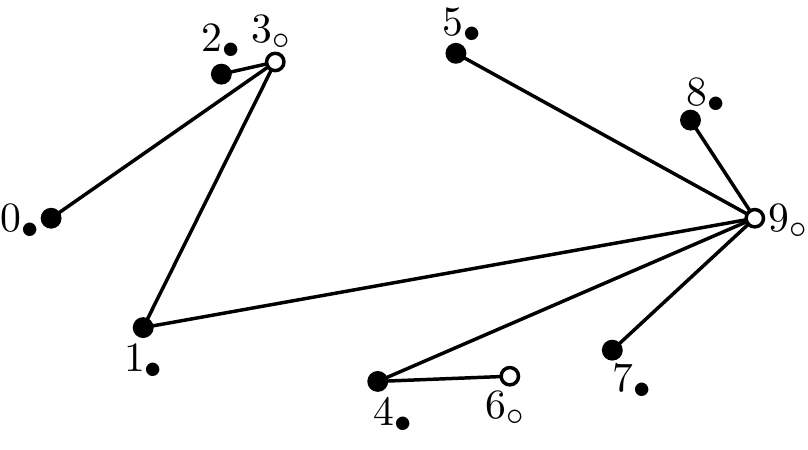} \\
	$g(\tree_{\min})_{3_\circ} = -1$ & $g(\tree)_{3_\circ} = \sqrt{3}-\sqrt{2}$ & $g(\tree_{\max})_{3_\circ} = \sqrt{2}-\sqrt{3}$ \\
	$g(\tree_{\min})_{6_\circ} = \sqrt{3}-1$ & $g(\tree)_{6_\circ} = \sqrt{3}-1$ & $g(\tree_{\max})_{6_\circ} = -\sqrt{2}$
	\end{tabular}
	}
	\caption{Examples of computation of coordinates.}
	\label{fig:coordinates}
\end{figure}

\end{example}

To conclude, let us gather all geometric realizations of the $(\signature, I_\bullet, J_\circ)$-lattice encountered in this paper (see Corollaries~\ref{coro:regionCambrianFan}, \ref{coro:dualTriangulation} and~\ref{coro:mixedSubdivision}, and Theorem~\ref{thm:tropical}).

\begin{theorem}
The increasing flip graph on~$(\signature, I_\bullet, J_\circ)$-trees can be realized geometrically as:
\begin{enumerate}
\item the dual of the collection of cones of the $\signature$-Cambrian fan of~\cite{ReadingSpeyer}, or of normal cones of the $\signature$-associahedron of~\cite{HohlwegLange}, corresponding to an interval of the $\signature$-Cambrian lattice,
\item the dual of a flag regular triangulation of the subpolytope~$\subpolyIJ$ of a product of simplices, 
\item the dual of a coherent fine mixed subdivision of a generalized permutahedron, 
\item the edge graph of a polyhedral complex defined by a tropical hyperplane arrangement. 
\end{enumerate}
\end{theorem}


\addtocontents{toc}{\vspace{.3cm}}
\section*{Acknowledgements}

I am grateful to C.~Ceballos, A.~Padrol and C.~Sarmiento for relevant comments and suggestions on the content of this paper.


\bibliographystyle{alpha}
\bibliography{tropicalCambrian}

\begin{thebibliography}{MHPS12}

\bibitem[Ber13]{Bergeron-multivariateDiagonalCoinvariantSpaces}
Fran\c{c}ois Bergeron.
\newblock Multivariate diagonal coinvariant spaces for complex reflection
  groups.
\newblock {\em Adv. Math.}, 239:97--108, 2013.

\bibitem[BM18]{BarnardMcConville}
Emily Barnard and Thomas McConville.
\newblock Lattices from graph associahedra and subalgebras of the
  malvenuto-reutenauer algebra.
\newblock Preprint,
  \href{http://arxiv.org/abs/1808.05670}{\texttt{arXiv:1808.05670}}, 2018.

\bibitem[BPR12]{BergeronPrevilleRatelle}
Fran\c{c}ois Bergeron and Louis-Fran\c{c}ois Pr\'eville-Ratelle.
\newblock Higher trivariate diagonal harmonics via generalized {Tamari} posets.
\newblock {\em Journal of Combinatorics}, 3(3):317--341, 2012.

\bibitem[CP17]{ChatelPilaud}
Gr\'egory Chatel and Vincent Pilaud.
\newblock {C}ambrian {H}opf {Algebras}.
\newblock {\em Adv. Math.}, 311:598--633, 2017.

\bibitem[CPS18]{CeballosPadrolSarmiento}
Cesar Ceballos, Arnau Padrol, and Camilo Sarmiento.
\newblock Geometry of $\nu$-{T}amari lattices in types {$A$} and {$B$}.
\newblock {\em Trans. Amer. Math. Soc.}, 2018.

\bibitem[DRS10]{DeLoeraRambauSantos}
Jesus~A. {De Loera}, J\"org Rambau, and Francisco Santos.
\newblock {\em Triangulations: Structures for Algorithms and Applications},
  volume~25 of {\em Algorithms and {C}omputation in Mathematics}.
\newblock Springer Verlag, 2010.

\bibitem[DS04]{DevelinSturmfels}
Mike Develin and Bernd Sturmfels.
\newblock Tropical convexity.
\newblock {\em Doc. Math.}, 9:1--27, 2004.

\bibitem[FZ02]{FominZelevinsky-ClusterAlgebrasI}
Sergey Fomin and Andrei Zelevinsky.
\newblock Cluster algebras. {I}. {F}oundations.
\newblock {\em J. Amer. Math. Soc.}, 15(2):497--529, 2002.

\bibitem[FZ03]{FominZelevinsky-ClusterAlgebrasII}
Sergey Fomin and Andrei Zelevinsky.
\newblock Cluster algebras. {II}. {F}inite type classification.
\newblock {\em Invent. Math.}, 154(1):63--121, 2003.

\bibitem[HL07]{HohlwegLange}
Christophe Hohlweg and Carsten Lange.
\newblock Realizations of the associahedron and cyclohedron.
\newblock {\em Discrete Comput.~Geom.}, 37(4):517--543, 2007.

\bibitem[HRS00]{HuberRambauSantos}
Birkett Huber, J\"org Rambau, and Francisco Santos.
\newblock The {C}ayley trick, lifting subdivisions and the {B}ohne-{D}ress
  theorem on zonotopal tilings.
\newblock {\em J. Eur. Math. Soc. (JEMS)}, 2(2):179--198, 2000.

\bibitem[LP18]{LangePilaud}
Carsten Lange and Vincent Pilaud.
\newblock Associahedra via spines.
\newblock {\em Combinatorica}, 38(2):443--486, 2018.

\bibitem[MHPS12]{TamariFestschrift}
Folkert M{\"u}ller-Hoissen, Jean~Marcel Pallo, and Jim Stasheff, editors.
\newblock {\em Associahedra, {T}amari Lattices and Related Structures. Tamari
  Memorial Festschrift}, volume 299 of {\em Progress in Mathematics}.
\newblock Springer, New York, 2012.

\bibitem[Pos09]{Postnikov}
Alexander Postnikov.
\newblock Permutohedra, associahedra, and beyond.
\newblock {\em Int. Math. Res. Not. IMRN}, (6):1026--1106, 2009.

\bibitem[PRV17]{PrevilleRatelleViennot}
Louis-Fran\c{c}ois Pr\'eville-Ratelle and Xavier Viennot.
\newblock An extension of {Tamari} lattices.
\newblock {\em Trans. Amer. Math. Soc.}, 369(7):5219--5239, 2017.

\bibitem[PRW08]{PostnikovReinerWilliams}
Alexander Postnikov, Victor Reiner, and Lauren~K. Williams.
\newblock Faces of generalized permutohedra.
\newblock {\em Doc.~Math.}, 13:207--273, 2008.

\bibitem[Rea06]{Reading-CambrianLattices}
Nathan Reading.
\newblock Cambrian lattices.
\newblock {\em Adv.~Math.}, 205(2):313--353, 2006.

\bibitem[RS09]{ReadingSpeyer}
Nathan Reading and David~E. Speyer.
\newblock Cambrian fans.
\newblock {\em J.~Eur.~Math.~Soc.}, 11(2):407--447, 2009.

\bibitem[Tam51]{Tamari}
Dov Tamari.
\newblock {\em Monoides pr\'eordonn\'es et cha\^ines de Malcev}.
\newblock PhD thesis, Universit\'e Paris Sorbonne, 1951.

\end{thebibliography}
\label{sec:biblio}

\end{document}